%% file: prox-VI.tex
\documentclass[11pt, english]{article}
\usepackage{times}
\usepackage[T1]{fontenc}
\usepackage[utf8]{inputenc}
\usepackage{babel}
\usepackage{amsfonts}
\usepackage{amsmath}
\usepackage{amsthm}
\usepackage{amssymb}
\usepackage[dvipsnames]{xcolor}
\usepackage{url}
\usepackage{centernot}
\usepackage{algorithm, algorithmic}
\renewcommand{\algorithmiccomment}[1]{\bgroup \hfill //~#1 \egroup}
\usepackage{mathtools}
\usepackage{mathabx}
\usepackage{hyperref}
\hypersetup{
	colorlinks = true,
	linkcolor= blue
}

\usepackage[margin=0.7in]{geometry}
\usepackage{accents}
\DeclareMathAccent{\wtilde}{\mathord}{largesymbols}{"65}

\usepackage{graphicx}
\usepackage{tikz-cd} 
\usepackage{changes}

\newtheorem{theorem}{Theorem}[section]
\newtheorem{proposition}{Proposition}[section]
\newtheorem{corollary}{Corollary}[section]

\newtheorem{lemma}[theorem]{Lemma}

\theoremstyle{definition}
\newtheorem{definition}[theorem]{Definition}

\theoremstyle{remark}
\newtheorem{remark}[theorem]{Remark}

\numberwithin{equation}{section}

\renewcommand{\textcolor}[1]{}
\renewcommand{\added}{}

\begin{document}
\include{macros}

\title{\bf 
	Optimal Algorithms for 
	Differentially Private  Stochastic \\ Monotone Variational Inequalities and Saddle-Point Problems}

\author{ 
Digvijay Boob \thanks{Engineering Management, Information, and Systems, Southern Methodist University, \texttt{dboob@smu.edu}}
\and
Crist\'{o}bal Guzm\'{a}n \thanks{Department of Applied Mathematics, University of Twente. Institute for Mathematical and Computational Eng., Pontificia Universidad Cat\'olica de Chile. \texttt{c.guzman@utwente.nl}}
}

\maketitle

\begin{abstract}
In this work, we conduct the first systematic study of stochastic variational inequality (SVI) and stochastic saddle point (SSP) problems under the constraint of differential privacy (DP). We propose two algorithms: Noisy Stochastic Extragradient (NSEG) and Noisy Inexact Stochastic Proximal Point (NISPP). \added{We show that a stochastic approximation variant of these algorithms attains risk bounds vanishing as a function of the dataset size, with respect to the strong gap function; and a sampling with replacement variant achieves optimal risk bounds with respect to a weak gap function. We also show lower bounds of the same order on weak gap function. Hence, our algorithms are optimal.} 
Key to our analysis is the investigation of algorithmic stability bounds, both of which are new even in the nonprivate case. 
The dependence of the running time of the sampling with replacement algorithms, with respect to the dataset size $n$, is $n^2$ for NSEG and $\wt{O}(n^{3/2})$ for NISPP. 
\end{abstract}

\input{sec_Introduction}

\input{sec_prelim_DP}

\input{sec_NSEG_convergence}

\input{sec_NSEG_stability_priv}

\input{sec_prox_convergence}

\input{sec_prox_conv_SPP}
\input{sec_stab_prox}

\input{sec_LB}

\section{Acknowledgements}
CG would like to thank Roberto Cominetti for valuable discussions on stochastic variational inequalities and nonexpansive iterations. CG's research is partially supported by INRIA through the INRIA Associate Teams project, ANID -- Millenium Science Initiative Program -- NCN17\_059, and FONDECYT 1210362 project. We would also like to thank the authors of \cite{Yang:2022} for pointing out a gap in a proof of an earlier version of this paper.

\bibliographystyle{plain}
\bibliography{bibliography}

\input{appendix}
\end{document}

%% file: macros.tex
\DeclarePairedDelimiter\abs{\lvert}{\rvert}
\DeclarePairedDelimiter{\bracket}{ [ }{ ] }
\DeclarePairedDelimiter{\paran}{(}{)}
\DeclarePairedDelimiter{\braces}{\lbrace}{\rbrace}
\DeclarePairedDelimiterX{\gnorm}[3]{\lVert}{\rVert_{#2}^{#3}}{#1}
\DeclarePairedDelimiter{\floor}{\lfloor}{\rfloor}
\DeclarePairedDelimiter{\ceil}{\lceil}{\rceil}
\DeclarePairedDelimiterX{\inprod}[2]{\langle}{\rangle}{#1, #2}
\DeclarePairedDelimiterX{\inprodo}[3]{\langle}{\rangle_{#3}}{#1, #2}
\providecommand{\dx}[2]{\frac{d #1}{d #2}}
\providecommand{\dxx}[2]{\frac{d^2 #1}{d {#2}^2}}
\providecommand{\dxy}[3]{\frac{d^2 #1}{d {#2} d{#3}}}
\providecommand{\dox}[2]{\frac{\partial #1}{\partial #2}}
\providecommand{\doxx}[2]{\frac{\partial^2 #1}{\partial {#2}^2}}
\providecommand{\doxy}[3]{\frac{\partial^2 #1}{\partial {#2} \partial{#3}}}
\providecommand{\at}[3]{\left.#1\right\vert_{#2}^{#3}}
\providecommand{\brr}[1]{{\left(#1\right)}} 
\providecommand{\sbr}[1]{{\ast\left(#1\right)}}
\newcommand\grad{\nabla}
\newcommand{\wt}[1]{{\widetilde{#1}}}
\newcommand{\uwt}[1]{{\underaccent{\wtilde}{#1}}}
\newcommand{\wb}[1]{{\widebar{#1}}}
\newcommand{\wh}[1]{{\widehat{#1}}}
\newcommand{\argmin}{\operatornamewithlimits{argmin}}
\newcommand{\argmax}{\operatorname{argmax}}
\newcommand{\dist}{\operatorname{dist}}
\newcommand{\mleq}[1]{\stackrel{\mathclap{\mbox{\normalfont \scriptsize #1}}}{\leq}}
\newcommand{\tsum}{\textstyle{\sum}}
\newcommand{\bP}{\mathbf{P}}
\newcommand{\Prob}{\mathbb{P}}

\newcommand{\batch}{\boldsymbol{\beta}}
\newcommand{\tw}{\widetilde{w}}

\newcommand{\M}{\widetilde M}
\newcommand{\T}[3]{R_{#1}(#3; #2)}
\newcommand{\DPPP}{DP-Proximal Point}
\newcommand{\gap}{\text{Gap}}
\newcommand{\gen}{\text{gen}}
\newcommand{\opt}{\text{opt}}
\newcommand{\appx}{\text{approx}}
\newcommand{\rsk}{\text{risk}}

\newcommand{\VIgap}[2]{\mbox{Gap}_{\mbox{\footnotesize VI}}(#1,#2)}
\newcommand{\EmpVIgap}[2]{\mbox{EmpGap}_{\mbox{\footnotesize VI}}(#1,#2)}
\newcommand{\wVIgap}[2]{\mbox{WeakGap}_{\mbox{\footnotesize VI}}(#1,#2)}
\newcommand{\SPgap}[2]{\mbox{Gap}_{\mbox{\footnotesize SP}}(#1,#2)}
\newcommand{\EmpSPgap}[2]{\mbox{EmpGap}_{\mbox{\footnotesize SP}}(#1,#2)}
\newcommand{\wSPgap}[2]{\mbox{WeakGap}_{\mbox{\footnotesize SP}}(#1,#2)}
\newcommand{\RSPgap}[3]{\mbox{RGap}_{\mbox{\footnotesize SP}}((#1,#2),#3)}
\newcommand{\GenP}[2]{\mbox{Gen}_{\mbox{\footnotesize P}}(#1;#2)}
\newcommand{\GenD}[2]{\mbox{Gen}_{\mbox{\footnotesize D}}(#1;#2)}
\newcommand{\SPgapP}[2]{\mbox{Gap}_{\mbox{\footnotesize P}}(#1;#2)}
\newcommand{\SPgapD}[2]{\mbox{Gap}_{\mbox{\footnotesize D}}(#1;#2)}

\newcommand{\knseg}{NSEG}
\newcommand{\pgap}{\boldsymbol{\eps}_\gen(1)}
\newcommand{\dgap}{\boldsymbol{\eps}_\gen(2)}
\newcommand{\priva}{\varepsilon}
\newcommand{\privb}{\eta}
\newcommand{\RDPa}{\alpha}
\newcommand{\RDPb}{\beta}
\newcommand{\SVI}{\mbox{SVI}}
\newcommand{\MVI}{\mbox{VI}}

\newcommand{\AbatchEG}{{\cal A}_{\mbox{\tiny batch-EG}}}
\newcommand{\AreplEG}{{\cal A}_{\mbox{\tiny repl-EG}}}
\newcommand{\AnoisyEG}{{\cal A}_{\mbox{\tiny noisy-EG}}}
\newcommand{\AreplNISPP}{{\cal A}_{\mbox{\tiny repl-NISPP}}}
\newcommand{\RR}{\mathbb{R}}
\newcommand{\PP}{\mathbb{P}}
\newcommand{\QQ}{\mathbb{Q}}
\newcommand{\NN}{\mathbb{N}}
\newcommand{\EE}{\mathbb{E}}
\newcommand{\bE}{\mathbb{R}^d}
\newcommand{\bB}{\mathbf{B}}
\newcommand{\Z}{\mathcal{Z}}
\newcommand{\W}{\mathcal{W}}
\newcommand{\U}{\mathcal{U}}
\newcommand{\eps}{\varepsilon}
\newcommand{\sphere}{\mathbb{S}}
\newcommand{\diag}{\mbox{diag}}
\newcommand{\bb}{\mathbf{b}}
\newcommand{\bbeta}{\boldsymbol{\beta}}
\newcommand{\bw}{\mathbf{w}}
\newcommand{\bW}{\boldsymbol{W}}
\newcommand{\bS}{\mathbf{S}}
\newcommand{\bR}{\mathbf{R}}
\newcommand{\bU}{\mathbf{U}}
\newcommand{\bJ}{\mathbf{J}}
\newcommand{\bI}{\mathbf{I}}
\newcommand{\bu}{\mathbf{u}}
\newcommand{\bv}{\mathbf{v}}
\newcommand{\bx}{\boldsymbol{x}}
\newcommand{\bz}{\mathbf{z}}
\newcommand{\bg}{\boldsymbol{\xi}}
\newcommand{\br}{\boldsymbol{r}}
\newcommand{\by}{\boldsymbol{y}}
\newcommand{\tr}{\mbox{tr}}
\newcommand{\op}{\mbox{\footnotesize op}}
\newcommand{\Q}{\mathcal{Q}}
\newcommand{\D}{\mathcal{D}}
\newcommand{\K}{\mathcal{K}}
\newcommand{\X}{\mathcal{X}}
\newcommand{\Y}{\mathcal{Y}}
\newcommand{\A}{\mathcal{A}}
\newcommand{\Pcal}{\mathcal{P}}
\newcommand{\prox}{\mbox{prox}}
\newcommand{\proj}{\Pi}
\newcommand{\KM}{Krasnosel'skii-Mann }
\newcommand{\cnote}[1]{ [\textcolor{RoyalBlue}{Crist\'obal: #1}] }
\newcommand{\dnote}[1]{ [\textcolor{Red}{Digvijay: #1}] }
\newcommand{\set}[1]{ \{ #1\} }

%% file: sec_Introduction.tex
\section{Introduction} \label{sec:intro}

Stochastic variational inequalities (SVI) and stochastic saddle-point (SSP) problems have become a central part of 
the modern machine learning toolbox. The main motivation behind this line of research is 
the design of algorithms for multiagent systems and adversarial training, which are more suitably modeled by 
the language of games, rather than pure (stochastic) optimization. Applications that rely on these methods 
may often involve the use of sensitive user data, so it becomes important to develop algorithms for these
problems with provable privacy-preserving guarantees. In this context, differential privacy (DP) has become 
the gold standard of privacy-preserving
algorithms, thus a natural question is whether it is possible to design DP algorithms for SVI and SSP that attain high
accuracy.

Motivated by these considerations, this work provides the first systematic study of differentially-private SVI and SSP 
problems. Before proceeding
to the specific results, we present more precisely the problems of interest. The stochastic variational inequality (SVI) 
problem is: given \textcolor{red}{a monotone} 
operator $F:\W\mapsto\RR^d$ in expectation form $F(w)=\EE_{\bbeta\sim {\cal P}}[F_{\bbeta}(w)]$, find 
$w^* \in \W$ such that
\begin{equation}\label{main_problem}
\inprod{F(w^*)}{w-w^*} \ge 0 \quad \forall w \in \W. \tag{VI(F)}
\end{equation}
The closely related stochastic saddle point (SSP) problem is: given a \textcolor{red}{convex-concave} real-valued function $f:{\cal W}\mapsto\RR$
(here ${\cal W}={\cal X}\times{\cal Y}$ is a product space),
given in expectation form  $f(x,y) = \EE_{\bbeta\sim\Pcal}[f_{\bbeta}(x, y)]$, the goal is to find $(x^{\ast},y^{\ast})$
that solves
\begin{equation}\label{spp_problem}
	\min_{x \in \X} \max_{y \in \Y}  f(x,y) . \tag{SP(f)}
\end{equation} 
In both of these problems, the input to the algorithm is an i.i.d.~sample $\bS=(\bbeta_1,\ldots,\bbeta_n)\sim {\cal P}^n$. 
Uncertainty introduced by a finite random sample renders the computation of exact solutions infeasible, so gap (a.k.a. population risk) 
functions are used to quantify the quality of solutions. 
\added{Let ${\cal A}:{\cal Z}^n\mapsto{\cal W}$ be an algorithm for SVI problems  \eqref{main_problem}.
	We define the {\em strong VI-gap} associated with ${\cal A}$ as
\begin{equation} \label{eqn:Minty_VI}
	\VIgap{{\cal A}}{F}:= \added{\EE_{{\cal A},\bS}}\bracket[\Big]{\sup_{w\in\W}\langle F(w),{\cal A}(\bS)-w \rangle}.
\end{equation}
We also define the {\em weak VI-gap} as
\begin{equation} \label{eqn:weak_VI_gap}
	\wVIgap{{\cal A}}{F} := \EE_{\cal A} \sup_{w \in \W} \EE_{\bS}\bracket*{\inprod{F(w)}{{\cal A}(\bS) -w}}.
\end{equation}
Here, expectation is taken over both the sample data $\bS$ and the internal randomization of ${\cal A}$.
For SSP \eqref{spp_problem}, given an algorithm ${\cal A}:{\cal Z}^n\mapsto {\cal X}\times{\cal Y}$, and letting ${\cal A}(\bS)=(x(\bS),y(\bS))$, a natural gap function is the following {\em saddle-point (a.k.a.~primal-dual) gap} 
\begin{equation} \label{eqn:SP_gap}
\SPgap{{\cal A}}{f}:= \EE_{{\cal A},\bS}\bracket[\Big]{\sup_{x\in\X,\, y\in \Y} [f(x(\bS), y)-f(x,y(\bS))] }.
\end{equation}
Analogously as above, we define the {\em weak SSP gap} as}\footnote{The denominations of weak and strong gap functions used in this paper are not standard, but we believe are the most appropriate in this context. For example, in \cite{Zhang:2020} used the terms {\em weak and strong generalization measure} for \eqref{eqn:weak_SP_gap} and \eqref{eqn:SP_gap} respectively, but it is clear that these quantities do not refer to standard generalization measures used in stochastic optimization.} 
\added{
\begin{equation} \label{eqn:weak_SP_gap}
	\wSPgap{{\cal A}}{f}:= \EE_{\cal A} \sup_{x\in\X,\, y\in \Y} \EE_{\bS}[f(x(\bS), y)-f(x,y(\bS))] .  
\end{equation}
}
It is easy to see that in both cases the gap is always nonnegative, and any exact solution must have zero-gap. For examples and applications of SVI and SSP we refer to Section \ref{sec:ex_app_SVI_SSP}.
\added{Despite the fact that the strong VI is a more classical and well-studied quantity, the weak VI gap has been observed to be useful in various contexts. We refer the reader to \cite{Zhang:2020} for more discussions on the weak VI gap.}

On the other hand, we are interested in designing algorithms that are {\em differentially private}. These algorithms build a solution
based on a given dataset $\bS$ of random i.i.d.~examples from the target distribution, and output a (randomized) feasible solution, ${\cal A}(\bS)$. 
We say that two datasets $\bS=(\bbeta_i)_i,\bS^{\prime}=(\bbeta_i^{\prime})_i$ are neighbors, denoted $\bS\simeq \bS^{\prime}$, if they only differ in a single entry $i$. We say that an algorithm ${\cal A}(\bS)$ is $(\priva,\privb)$-differentially private if for every event $E$ in the output 
space\footnote{Note that the probabilities in the definition of DP only involve the probability space of algorithmic randomization, and not of the datasets, which is emphasized by the notation $\PP_{\cal A}$. The datasets must be neighbors, but they are 
otherwise arbitrary, and this is crucial to certify the privacy for any user.}
\begin{equation} \label{eqn:DP} 
\PP_{\cal A}[{\cal A}(\bS)\in E] \leq e^{\priva}\PP_{\cal A}[{\cal A}(\bS^{\prime})\in E] + \privb
\quad(\forall \bS\simeq \bS^{\prime}).
\end{equation}
Here $\priva,\privb\geq 0$ are prescribed parameters that quantify the privacy guarantee. Designing DP algorithms for 
particular data analysis problems is an active area of research. Optimal risk algorithms 
for stochastic convex optimization have only very recently been developed, and it is unclear whether these methods
are extendable to SVI and SSP settings.

\subsection{Summary of Contributions}
Our work is the first to provide population risk bounds for DP-SVI and DP-SSP problems. Moreover, our algorithms attain
provably optimal rates and are computationally efficient. We summarize our contributions as follows:

\begin{enumerate}
\item We provide two different algorithms for DP-SVI and DP-SSP: namely, the noisy stochastic extragradient method
(\knseg) and a noisy inexact stochastic proximal-point method (NISPP). The \knseg\, method is a natural DP variant 
of the well-known stochastic extragradient method \cite{Juditsky:2011}, where privacy is obtained by Gaussian noise 
addition; on the other hand, the NISPP method is an approximate proximal point algorithm  
\cite{Rockafellar:1976,Jalilzadeh:2019} in which every proximal iterate is made noisy to make it differentially private. Our more basic variants of both of these methods are based
on \added{iterations involving disjoint sets of datapoints (a.k.a.~single pass method)}, which are known to typically lead to highly suboptimal rates in DP (see the Related Work
Section for further discussion). 
\item We derive novel uniform stability bounds for the \knseg\, and NISSP methods.
For \knseg, our stability upper bounds are inspired by the interpretation of the extragradient method as a (second order) 
approximation of the proximal point algorithm. In particular, we provide expansion bounds for the extragradient iterates, and 
solve a (stochastic) linear recursion. The stability bounds for NISPP method are based on stability of the (unique) SVI solution
in the  strongly monotone case. Finally, we investigate the risk attained by multipass versions of the \knseg\, and NISPP methods, leveraging \added{known generalization bounds for stable algorithms \cite{Lei2021stability}.}
Here, we show that 
the optimal risk for DP-SVI and DP-SSP can be attained by running these 
algorithms with their sampling with replacement variant. In particular, \knseg\ method requires $n^2$ stochastic operator evaluations, and NISPP method requires much smaller $\wt{O}(n^{3/2})$ operator evaluations for both DP-SVI and DP-SSP problems. In particular, these upper bounds also show the dependence of the running time of each of these algorithms w.r.t.~the dataset size.
\item \added{Finally, we prove lower bounds on the weak gap function for any DP-SSP and DP-SVI algorithm, showing the risk optimality of the aforementioned multipass algorithms. The main challenge in these lower bounds is showing that existing constructions of lower bounds for DP convex optimization \cite{Bassily:2014,Steinke:2016,Bassily:2019} lead to lower bounds on the weak gap of a related SP/VI problem.}
\end{enumerate}
\textcolor{red}{
The following table provides details of population risk and operator evaluation complexity .
\begin{table}[H]
	\centering
	\begin{tabular}{|c|c|c|c|c|}\hline
		&\multicolumn{2}{c|}{Type of sampling} &\multicolumn{2}{c|}{Type of sampling}\\\hline
		&single pass &multipass&single pass &multipass\\\hline
		Method&Criterion - Strong Gap &Criterion - Weak Gap &\multicolumn{2}{c|}{Number of operator evaluations}\\\hline
		NSEG &$O\Big(\frac{d^{1/4}}{\sqrt{n\varepsilon}}+ \frac{\sqrt{d}}{n\varepsilon} \Big)$ &$O\Big(\frac{1}{\sqrt{n}} + \frac{\sqrt{d}}{n\varepsilon} \Big)$ &$n$ & $n^2$\\[1.5mm]\hline 
		NISPP (OE subroutine) &$O\Big(\frac{1}{n^{1/3}}+ \frac{\sqrt{d}}{ n^{2/3}\varepsilon}\Big)$ &$O\Big(\frac{1}{\sqrt{n}} + \frac{\sqrt{d}}{n\varepsilon}\Big)$ &$O(n\log{n})$ &$O(n^{3/2}\log{n})$\\\hline
	\end{tabular}
	\caption{Different levels of risk and complexity achieved  by NSEG and NISPP methods for $(\varepsilon,\eta)$-differentially private SVI/SSP. Here $n$ is the dataset size, and $d$ is the dimension of the solution search space. We omit the dependence on other problem parameters (e.g., Lipschitz constants and diameter), as well as the privacy parameter $\eta$.}
\end{table}
}

\subsection{Related Work}

We divide our discussion on related work in three main areas. Each of these areas has been extensively investigated, so 
a thorough description of existing work is not possible. We focus ourselves on the work which is more directly related to our own.

\begin{enumerate}
\item {\bf Stochastic Variational Inequalities and Saddle-Point Problems:} Variational inequalities and saddle-point problems are classical topics in applied mathematics, operations research and engineering (e.g., \cite{vonNeumann:1928,Sion:1958,Rockafellar:1976,Korpelevich:1976,Nemirovsky:1983,Facchinei:2007,Nemirovski:2004,Auslender:2005,Nesterov:2005}). Their stochastic 
counterparts have only gained traction recently, mainly motivated by their applications in machine learning 
(e.g.,\cite{Juditsky:2011,Juditsky:2012,Iusem:2019,Hsieh:2019,Kotsalis:2020} and references therein).
For the stochastic version of \eqref{spp_problem}, \cite{Nemirovski:2009} proposed a robust stochastic approximation method. 
The first optimal algorithm for SVI with monotone Lipschitz operators was obtained by Juditsky, Nemirovski and Tauvel 
\cite{Juditsky:2011}, and very recently Kotsalis, Lan and Li \cite{Kotsalis:2020} developed optimal variants for the strongly monotone case
(in terms of distance to the optimum criterion, rather than VI gap).

It is important to note that naive adaptation of these methods to the DP setting requires adding noise to the operator 
evaluations at every iteration, which substantially degrades the accuracy of the obtained solution. A careful privacy accounting
and minibatch schedule can lead to optimal guarantees for single-pass methods \cite{Feldman:2020}, however
this requires accuracy guarantees for the last iterate, which is currently an open problem for SVI and SSP 
(aside from specific cases, typically involving strong monotonicity conditions, e.g., \cite{Hsieh:2019,Kotsalis:2020}).
We circumvent this problem by providing population risk guarantees for {\em multipass methods}.

\item {\bf Stability and Generalization:} Deriving generalization (or population risk) bounds for general-purpose algorithms is 
a challenging task, actively studied in theoretical machine learning. Bousquet and Elisseeff \cite{Bousquet:2002} provided a 
systematic treatment of this question for algorithms which are {\em stable}, with respect to changes of a single element
in the training dataset, and a sequence of works have refined these generalization guarantees (see \cite{FeldmanVondrak:2019,Bousquet:2020} and references therein). This idea has been applied to investigate the generalization properties of regularized empirical risk
minimization \cite{Bousquet:2002,SSSSS:2010}, and more recently to iterative methods, such as stochastic gradient descent
\cite{Hardt:2016,BFGT:2020}.

Using stability to obtain population risk bounds in SVI and SSP is substantially more challenging, due to the presence of a 
supremum in the accuracy measure (see eqns.~\eqref{eqn:Minty_VI} and \eqref{eqn:SP_gap}). Recently, Zhang et al. 
\cite{Zhang:2020}, established stability implies generalization results for the strong SP gap 
\added{under strong monotonicity assumptions. 
Their proof strategy applies analogously to address the SVI setting,
although this is not carried out in their work. More recently, Lei et al.~\cite{Lei2021stability}, proved generalization bounds on the weak
SP gap without strong monotonicity assumptions. We leverage this result for our
algorithms, and further elaborate on its implications for SVI in Section \ref{sec:algorithmic_stab}.
}



\item {\bf Differential Privacy in (Stochastic) Convex Optimization:} Differentially private empirical risk minimization and stochastic convex
optimization have been extensively studied for over a decade (see, e.g.~\cite{CM08,CMS,jain2012differentially,kifer2012private,Bassily:2014,ullman2015private,JTOpt13,Bassily:2019,Feldman:2020}). Relevant to our work are the first optimal risk algorithms for 
DP-ERM \cite{Bassily:2014} and DP-SCO \cite{Bassily:2019}. To the best of our knowledge, our work is the first to address
DP algorithms for SVI and SSP. \added{Our approach for generalization of multipass algorithms is inspired by the noisy SGD analysis in \cite{BFGT:2020}. However, 
our stability analysis differs crucially from \cite{BFGT:2020}: in the case of NSEG, we need to carefully address the double operator evaluation of the extragradient step, which is done by using the fact that the extragradient operator is approximately nonexpansive. In the case of NISPP, we leverage the contraction properties of strongly monotone VI solutions. By contrast, SGD in the nonsmooth case is far from nonexpansive \cite{BFGT:2020}.}
Alternative approaches to obtain  optimal risk in DP-SCO, including privacy amplification by iteration \cite{Feldman:2020}, and
phased regularization or phased SGD \cite{Feldman:2020}, appear to run into fundamental limitations when applied to DP-SVI 
and DP-SSP. It is an interesting future research direction to obtain faster running times with optimal population risk in DP-SVI and 
DP-SSP, which
may benefit from these alternative approaches.

\end{enumerate}

The main body of this paper is organized as follows. In Section \ref{sec:notation}, we provide the necessary background information on 
SVI/SSP, uniform stability, and differential privacy, which are necessary for the rest of the paper. In Section \ref{sec:nseg} we introduce the 
\knseg\, method, together with its basic privacy and accuracy guarantee for a single pass version.  
Section \ref{sec:stab_nseg} provides stability bounds for NSEG method along with the consequent optimal rates for SVI and SSP. In Section \ref{sec:prox_DPSVI}, we introduce the single-pass differentially private NISPP method with bound on expected SVI-gap. Section \ref{sec:stab_risk_NISPP} presents stability analysis of NISPP, together with the resulting optimal rates for 
SVI/SSP gap. We conclude in Section \ref{sec:LowerBounds} with lower bounds that prove the optimality of the obtained rates.

\section{Notation and Preliminaries} \label{sec:notation}
We work on the Euclidean space $(\RR^d, \inprod{\cdot}{\cdot})$, where $\inprod{\cdot}{\cdot}$ is
the standard inner product, and $\|u\|=\sqrt{\langle u,u\rangle}$ is the $\ell_2$-norm. Throughout, we consider a compact convex set 
$\W\subseteq\RR^d$ with diameter $D>0$.  We denote the standard Euclidean projection operator on set $\W$ by $\proj_{\W}(\cdot)$. 
The identity matrix on $\RR^d$ is denoted by $\mathbb{I}_d$.

We let $\mathcal{P}$ denote an unknown distribution supported on an arbitrary set $\Z$, from which we have access to exactly $n$ i.i.d.~datapoints which we denote by sample set $\bS\sim {\cal P}^n$. Throughout, we will use boldface characters to denote sources of randomness (coming from the data, or internal algorithmic randomization).  
We say that two datasets $\bS, \bS'$ are adjacent (or neighbors), denoted by $\bS \simeq \bS'$, if they differ in a single data point. We also denote subsets (a.k.a. batches), or single 
data points,
of $\bS$ or $\Pcal$ by $\bB$ and $\bbeta$, respectively. Whether $\bbeta$ or $\bB$ is sampled from $\Pcal$ or $\bS$ is specified explicitly unless it is clear from the context. For a batch $\bB$, we denote its size by $\abs{\bB}$. Therefore, we have $\abs{\bS} = n$. Throughout, we will denote Gaussian random variables by $\boldsymbol{\xi}$. 

We say that $F:\W \to \RR^d$ is a monotone operator if 
\[\inprod{F(w_1) -F(w_2)}{w_1-w_2} \ge 0, \quad \forall w_1, w_2 \in \W.\]
Given $L>0$, we say that $F$ is  $L$-Lipschitz continuous, if
\[\gnorm{F(w_1) - F(w_2)}{}{} \le L\gnorm{w_1-w_2}{}{}, \quad \forall w_1, w_2 \in \W.\]
Finally, we say that $F$ is $M$-bounded if $\sup_{w\in \W}\|F(w)\|\leq M$. We denote the
set of monotone, $L$-Lipschitz and $M$-bounded operators by ${\cal M}_{\cal W}^1(M,L)$. 
In this work, we will focus on the case where $F$ is an expectation operator, i.e., $F(w) := \EE_{\bbeta\sim \Pcal} [F_{\bbeta}(w)]$, where $\Pcal$ is an arbitrary distribution supported on $\Z$, \textcolor{red}{
and for any $\bbeta$ in ${\cal Z}$, $F_{\bbeta}(\cdot)\in {\cal M}_{\cal W}^1(M,L)$}, $\bbeta$-a.s.\footnote{
\added{Here, we mean that for almost every $\bbeta$, we have $F_{\bbeta}\in {\cal M}_{\cal W}^1(M,L)$.} 
}


In the stochastic saddle point problem \eqref{spp_problem}, we modify the notation slightly. Here, ${\cal X}\subseteq\RR^{d_1}$ and ${\cal Y}\subseteq\RR^{d_2}$ are compact convex sets, and we will assume that the saddle point
functions $f_{\bbeta}(\cdot,\cdot):{\cal X}\times{\cal Y}\mapsto\RR$, satisfy the following conditions $\bbeta$-a.s. 
\begin{itemize}
	\item $\nabla_x f_{\bbeta}(\cdot,\cdot)$ is $L_x$-Lipschitz continuous, and $\nabla_y f_{\bbeta}(\cdot,\cdot)$ is $L_y$-Lipschitz continuous, and;
	\item $f_{\bbeta}(\cdot,y)$ is convex, for any given $y\in \Y$, and $f_{\bbeta}(x,\cdot)$ is concave, for any given $x\in \X$ (we will say in this case the function is convex-concave).
\end{itemize}
If the assumptions above are met, we will denote $L\triangleq \sqrt{L_x^2 + L_y^2}$.  
Under the assumptions above, it is well-known that SSP \cite{vonNeumann:1928,Sion:1958} (and 
SVI \cite{Facchinei:2007}, respectively) have a solution.

In the case of saddle-point problems, given the convex-concave function $f_{\bbeta}(\cdot,\cdot):{\cal X}\times{\cal Y}\mapsto\RR $, it is well-known that the operator $F:{\cal X}\times{\cal Y}\mapsto \RR^d\times\RR^d$ below is monotone
\begin{equation} \label{eqn:monotone_op_SP}
F_{\bbeta}(x,y) = (\nabla_x f_{\bbeta}(x,y), -\nabla_y f(x,y)).
\end{equation}
We will call this operator the {\em monotone operator associated with} $f_{\bbeta}(\cdot,\cdot)$. 
Furthermore, if $\nabla_x f_{\bbeta}(\cdot,y)$ has $L_x$-Lipschitz continuous gradient and $\nabla_y f_{\bbeta}(x,\cdot)$ has $L_y$-Lipschitz continuous gradient, then $F$ is $\sqrt{L_x^2+L_y^2}$-Lipschitz continuous. 

It is easy to see that, given a SSP problem with function $f_{\bbeta}(\cdot,\cdot)$ and sets ${\cal X}$, ${\cal Y}$, an (exact) SVI solution \eqref{main_problem} for the monotone operator associated to 
$f(x,y)=\EE_{\bbeta}[f_{\bbeta}(x,y)]$ over the set $\W={\cal X}\times{\cal Y}$, yields an exact SSP solution for the starting problem. Unfortunately, such reduction does not directly work for approximate solutions to \eqref{eqn:Minty_VI} and \eqref{eqn:SP_gap}, so the analysis must be done separately for both problems.
 
For batch $\bB$, we denote the empirical (a.k.a.~sample average) operator 
$F_{\bB}(w) := \frac{1}{\abs{\bB}} \sum_{\bbeta \in \bB} F_{\bbeta}(w).$ 
On the other hand, for a batch $\bB$, the empirical saddle point function is denoted as 
$f_{\bB}(x, y) = \frac{1}{\abs{\bB}} \sum_{\bbeta \in \bB} f_{\bbeta}(x, y).$
\added{Given a distribution ${\cal P}$, the expectation operator and function are denoted by $F_{\cal P}(w):=\mathbb{E}_{\bbeta\sim\cal P}[F_{\bbeta}(w)]$, and $f_{\cal P}(x,y)=\EE_{\bbeta\sim {\cal P}}[f_{\bbeta}(x,y)]$, respectively. 
For brevity, whenever it is clear from context we will drop the dependence on ${\cal P}$.}\\

\subsection{Examples and Applications of SVI and SSP} \label{sec:ex_app_SVI_SSP}
An interesting problem which can be formulated as a SSP-problem is the minimization of a max-type convex function:
\textcolor{red}{\[\min_{x \in \X} \Big\{\phi(x) :=\max_{1 \le j \le m} \phi_j(x)\Big\},\]}
where $\phi_j: \X \to \RR$ is a stochastic convex function $\phi_j(x) := \EE_{\boldsymbol{\zeta}_j\sim\Pcal_j}[\phi_{j,\boldsymbol{\zeta}_j}(x)]$ for all $j \in [m]$. This problem is essentially a structured nonsmooth optimization problem which can be reformulated into a convex-concave saddle point problem: 
\[\min_{x \in \X} \max_{y \in \Delta_m} \EE_{\boldsymbol{\zeta}_1\dots\boldsymbol{\zeta}_m} [\tsum_{j=1}^m y_i\phi_{j,\boldsymbol{\zeta}_j}(x)]\]
Here, $\bbeta = (\boldsymbol{\zeta}_j)_{j=1}^m$ is the random input to the saddle point problem: $f_{\bbeta}(x,y) = \sum_{j = 1}^my_j\phi_{j,\boldsymbol{\zeta}_j}(x)$. Note that a substantial generalization of the max-type problem above is the so called compositional optimization problem:
\[\min_{x \in \X}\phi(x) := \Phi(\phi_1(x), \dots, \phi_m(x)),\]
where $\phi_j(x)$ are convex maps and $\Phi(u_1, \dots, u_m)$ is a real-valued convex function whose Fenchel-type representation 
is assumed to have the form
\[\Phi(u_1, \dots, u_m) = \max_{y \in \Y} \tsum_{j=1}^m \inprod{u_j}{A_jy+ b_j} - \Phi_{*}(y),\]
where $\Phi_{\ast}$ is a convex, Lipschitz and smooth.
Then, overall optimization problem can be reformulated as a convex-concave saddle point problem:
\[\min_{x \in \X} \max_{y \in \Y} \tsum_{j=1}^m \inprod{\phi_j(x)}{A_jy+ b_j} - \Phi_{\ast}(y),\]
where stochasticity is introduced due to constituent functions $\phi_j(x) = \EE_{\boldsymbol{\zeta}_j} [\phi_{j,\boldsymbol{\zeta}_j}(x)].$ 

To conclude, we remark that these type of models have been recently proposed in machine learning to address {\em approximate
fairness} \cite{Williamson:2019} and federated learning on heterogeneous populations \cite{Li:2020}. In these examples, the different 
indices $j\in [m]$ may denote different subgroups from a population, and we are interested in bounding the (excess) population risk
on these subgroups uniformly (with the motivation of preventing discrimination against any subgroup). This clearly cannot be achieved 
by a stochastic convex program, and a stochastic saddle-point
formulation is effective in certifying accuracy across the different subgroups separately. 

For further examples and applications of stochastic variational inequalities and saddle-point problems, we refer the reader to \cite{Juditsky:2011,Juditsky:2012,Zhang:2020}.



\subsection{Algorithmic Stability} \label{sec:algorithmic_stab}

In general, an algorithm is a randomized function mapping datasets to candidate solutions, ${\cal A}:\Z^n\mapsto\RR^d$, \textcolor{red}{which is measurable w.r.t.~the dataset}. 
Two datasets, $\bS=(\bbeta_1,\ldots,\bbeta_n),\,\bS^{\prime}=(\bbeta_1^{\prime},\ldots,\bbeta_n^{\prime})\in\Z^n$ are said to be neighbors (denoted $\bS\simeq\bS^{\prime}$) if they only differ in at most one data point, namely
$$ \bbeta_j=\bbeta_j^{\prime} \qquad (\exists i\in [n])(\forall j\neq i). $$
Algorithmic stability is a notion of sensitivity analysis of an algorithm under neighboring datasets. Of particular interest to our work
is the notion of {\em uniform argument stability} (UAS).
\begin{definition}[Uniform Argument Stability] \label{def:UAS}
Let ${\cal A}:\Z^n\mapsto\RR^d$ be a randomized mapping and $\delta>0$. We say that ${\cal A}$ is $\delta$-uniformly argument stable (for short, $\delta$-UAS) if
$$ \sup_{\bS\simeq\bS^{\prime}} \EE_{\cal A}\|{\cal A}(\bS)-{\cal A}(\bS^{\prime})\|\leq \delta. $$
\end{definition} 
Ocassionally, we may denote $\delta_{\cal A}(\bS,\bS^{\prime})\triangleq \|{\cal A}(\bS)-{\cal A}(\bS^{\prime})\|$, for convenience.   
The importance of algorithmic stability in machine learning comes from the fact that {\em stability implies generalization} in stochastic optimization and stochastic saddle point (SSP) problems \cite{Bousquet:2002,Bousquet:2020,Zhang:2020}. \added{Below, we restate existing results on stability implies generalization for SSP problems below. Before doing so we need to briefly introduce the {\em (strong) empirical gap function}: given a dataset $\bS$ and an algorithm $\A$, we define the empirical
gap function for a saddle point and variational inequality problem respectively as
\begin{eqnarray} 
\EmpSPgap{\A}{f_{\bS}} &:=& \EE_{\cal A}[\sup_{x,y} f_{\bS}(x(\bS),y)-f_{\bS}(x,y(\bS)) ]\label{eqn:emp_SP_gap}\\
\EmpVIgap{\A}{F_{\bS}} &:=& \EE_{\cal A}[\sup_{w} \langle F_{\bS}(w),\A(\bS)-w\rangle]. \label{eqn:emp_VI_gap}
\end{eqnarray}
Notice that in these definitions the dataset $\bS$ is fixed.
\begin{proposition}\cite{Zhang:2020,Lei2021stability}\label{prop:weak_generalization_SSP}
Consider the stochastic saddle point problem \eqref{spp_problem} with functions $f_{\bbeta}(\cdot, y)$ and $f_{\bbeta}(x, \cdot)$ being $M$-Lipschitz for all $x \in\X, y \in \Y$ and $\bbeta$-a.s.. Let $\A: \Z^n \to \X \times \Y$ be an algorithm, where $\A(\bS)=(x(\bS),y(\bS))$.
If $x(\cdot)$ is $\delta_x$-UAS and $y(\cdot)$ is $\delta_y$-UAS, and both are integrable, then 
\begin{equation}\label{eq:stab_gen_ssp_final}
	\wSPgap{\A}{f} \le \EE_{\bS}[ \EmpSPgap{\A}{f_{\bS}} ] + M[\delta_x+\delta_y].
\end{equation}
\end{proposition}
This result can be extended for SVI problems as well. We provide a formal statement below and prove it in Appendix 
\begin{proposition}\label{prop:weak_generalization_SVI}
	Consider a stochastic variational inequality with $M$-bounded operators $F_{\bbeta}(\cdot):\W\mapsto\RR^d$. 
	If ${\cal A}:\Z^n\mapsto\W$ is integrable and $\delta$-UAS, then
	\begin{equation}
		\wVIgap{{\cal A}}{F}
		\leq 
		\EE_{\bS}[\EmpVIgap{{\cal A}}{F_\bS} ]
		+M\delta.
	\end{equation}
\end{proposition}
}

%% file: sec_prelim_DP.tex
\subsection{Background on differential privacy}
\label{sec:background_dp}
Differential privacy is an algorithmic stability type of guarantee for randomized algorithms, that certifies that the output distribution
of the algorithm ``does not change too much'' by changes in a single element from the dataset. The formal definition is
provided in eqn. \eqref{eqn:DP}.
Next we provide some basic results in differential privacy, which we will need for our work. For further information on the topic, 
we refer the reader to the monograph \cite{Dwork:2014}.

\subsubsection{Basic privacy guarantees}

In this work, most of our privacy guarantees will be obtained by the well-known {\em Gaussian mechanism},
which performs Gaussian noise addition on a function with bounded sensitivity. Given a function 
${\cal A}:\Z^n\mapsto \RR^d$, we define its $\ell_2$-sensitivity as
\begin{equation} \label{eqn:sensitivity}
\sup_{\bS\simeq \bS^{\prime}}\|{\cal A}(\bS)-{\cal A}(\bS^{\prime})\|.
\end{equation}
If ${\cal A}$ is randomized, then the supremum must hold with high-probability over the 
randomization of ${\cal A}$ (this will not be a problem in this work, since our randomized algorithms enjoy sensitivity bounds w.p.~1). 
The Gaussian mechanism (associated to function ${\cal A}$) is defined as 
${\cal A}_{\cal G}(\bS)\sim {\cal N}({\cal A}(\bS),\sigma^2I)$.

\begin{proposition} \label{propos:Gauss_mech}
Let ${\cal A}:\Z^n\mapsto \RR^d$ be a function with $\ell_2$-sensitivity $s>0$. Then, for 
$\sigma^2=2s^2\ln(1/\privb)/\priva^2$,  the Gaussian mechanism is $(\priva,\privb)$-DP.
\end{proposition}
Our algorithms will adaptively use a DP mechanism such as the above. Certifying privacy of a composition can be 
achieved in different ways. The most basic result establishes that if we use disjoint batches of data at each iteration, then the 
composition will preserve the largest privacy parameter among its building blocks. This result is known as 
{\em parallel composition}, and its proof its a direct application of 
the post-processing property of DP. 

\begin{proposition}[Parallel composition of differential privacy] \label{propos:parallel_comp}
Let $\bS=(\bS_1,\ldots,\bS_K)\in\Z^n$ be a dataset partitioned on blocks of sizes $n_1,\ldots,n_K$, respectively. 
${\cal A}_k:\Z^{n_k}\times \RR^{d\times (k-1)}\mapsto \RR^d$, $k=1,\ldots,K$, be a sequence of mechanisms, 
and let ${\cal A}:\Z^n\mapsto\RR^d$
be given by
\begin{eqnarray*}
{\cal B}_1(\bS) &=& {\cal A}_1(\bS_1)\\
{\cal B}_{k}(\bS) &=& {\cal A}_{k}(\bS_{k}, {\cal B}_{1}(\bS),{\cal B}_{2}(\bS),\ldots,{\cal B}_{k-1}(\bS)) \quad (\forall k=2,\ldots,K-1)\\
{\cal A}(\bS)            &=&{\cal A}_{K}(\bS_{K},{\cal B}_{1}(\bS),{\cal B}_{2}(\bS),\ldots,{\cal B}_{K-1}(\bS)).
\end{eqnarray*}
Then, 
If each ${\cal A}_k$ is $(\priva_k,\privb_k)$-DP in its first argument (i.e., w.r.t.~$\bS_k$) then ${\cal A}$ is $(\max_k \priva_k, \max_k \privb_k)$-DP.
\end{proposition}

Some of the algorithms we develop in this work make repeated use of the data, and certifying privacy for these 
algorithms requires the use of adaptive composition results in DP (see, e.g.~\cite{Dwork:2010,Dwork:2014}). For our algorithms, it 
is particularly important to leverage the sampling with replacement procedure to select the data that is used
at each iteration, for which sharp bounds on DP can be obtained by the {\em moments accountant method} \cite{Abadi:2016}. 
Below we summarize a specific version of this method that suffices for our purposes.\footnote{In our case we use uniform sampling on each iteration, as opposed to the Poisson sampling of \cite{Abadi:2016}; 
however, it is possible to verify that similar moment estimates lead to our stated result.}

\begin{theorem}[\cite{Abadi:2016}]\label{thm:dp_prox_multipass}
Consider sequence of functions ${\cal A}_1,\ldots,{\cal A}_K$, 
where ${\cal A}_k:\Z^{n_k}\times \RR^{d\times(k-1)}\mapsto\RR^d$ is a function with 
sensitivity bounded as a function of the last data batch size, as follows
$$ \sup_{L\in\RR^{d\times(k-1)},\bS_k\simeq\bS_k^{\prime}}\|{\cal A}_k(\bS_k, L)-{\cal A}_k(\bS_k^{\prime}, L)\|\leq {s}
.$$
Consider the mechanism obtained by sampling a random subset of size $m$ from the dataset, i.e., letting $\bS_k\sim (\mbox{Unif}([\bS]))^m$, and composing it with a Gaussian mechanism with noise $\sigma^2$, i.e.
\begin{eqnarray*}
{\cal B}_1(\bS) &=& ({\cal A}_1)_{\cal G}(\bS_{1}) \\
{\cal B}_{k}(\bS) &=& ({\cal A}_{k})_{\cal G}(\bS_{k}, {\cal B}_{1}(\bS),{\cal B}_{2}(\bS),\ldots,{\cal B}_{k-1}(\bS)) \quad (\forall k=2,\ldots,K).
\end{eqnarray*}
There exists an absolute constant $c_1>0$, such that if $\priva<c_1K(m/n)^2$ and the noise parameter $\sigma\geq \sqrt{2K\ln(1/\privb)}sm/[n\priva]$, then  $\A(\bS) := \{{\cal B}_1(\bS), \ldots, {\cal B}_K(\bS)\}$  is $(\priva,\privb)$-differentially private.
\end{theorem}

%% file: sec_NSEG_convergence.tex
\section{The Noisy Stochastic Extragradient Method}\label{sec:nseg}


To solve the DP-SVI problem we propose 
a noisy stochastic extragradient method (NSEG) in Algorithm \ref{alg:alg1}.
\begin{algorithm}[H]
	\caption{Noisy Stochastic Extragradient (NSEG) Method}
	\label{alg:alg1}
	\begin{algorithmic}[1]
		\STATE {\bf Input:} Starting point $u_0 \in \W$, dataset $\bS=(\bbeta_i)_{i\in[n]}\sim{\cal P}^n$, 
		stepsizes $(\gamma_t)_{t\in[T]}$
		\FOR {$t = 1, \dots, T$}
		\STATE $F_{1,t}(\cdot) = F_{\bB_t^1}(\cdot)+\boldsymbol{\xi}_t^1$, where $\bB_t^1\subseteq\bS$ and  $\boldsymbol{\xi}_t^1\sim {\cal N}(0,\sigma_t^2)$
		\STATE $w_t = \proj_{\W}(u_{t-1}-\gamma_t F_{1,t}(u_{t-1}))$ 
		\STATE $F_{2,t}(\cdot)= F_{\bB_t^2}(\cdot)+\boldsymbol{\xi}_t^2$, where $\bB_t^2\subseteq\bS$ and $\boldsymbol{\xi}_t^1\sim {\cal N}(0,\sigma_t^2)$
		\STATE $u_t = \proj_{\W}(u_{t-1}-\gamma_t F_{2,t}(w_t))$ 
		\ENDFOR
		\STATE{\bf return} 
		$\wb{w}^T =(\tsum_{t = 1}^T\gamma_t)^{-1} \tsum_{t = 1}^{T}\gamma_tw_t$
	\end{algorithmic}
\end{algorithm}
The name noisy and stochastic in Algorithm \ref{alg:alg1} is justified by the sequence of operators $F_{1,t},F_{2,t}$ we use:
\begin{eqnarray} \label{eqn:stoch_oracle_nseg}
	F_{1,t}(\cdot) \triangleq F_{\bB_t^1}(\cdot)+\boldsymbol{\xi}_t^1,
	&\quad& F_{2,t}(\cdot)\triangleq F_{\bB_t^2}(\cdot)+\boldsymbol{\xi}_t^2.
\end{eqnarray} 
where $\bB_t^1,\bB_t^2$ are batches extracted from dataset $\bS$, and $\boldsymbol{\xi}_t^1,\boldsymbol{\xi}_t^2\stackrel{i.i.d.}{\sim}{\cal N}(0,\sigma_t^2)$. We will denote the batch size of batch $\bB_t^1$ and $\bB_t^2$ as 
$B_t:=|\bB_t^1|=|\bB_t^2|$.
The exact details of the sampling method for $\bB_t$ will depend on the variant of the algorithm.
Here, we detail some key features of the above algorithm. Stochastic extragradient was proposed in \cite{Juditsky:2011} where they do not have any noise addition in $F_{1,t}, F_{2,t}$ (stochasticity only arises from the dataset randomness), and where disjoint batches 
are used for all iterations, as well as within iterations. This choice
is motivated by the goal of extracting population risk bounds for their algorithm. 

Another important consideration is that this algorithm can also be applied to an SSP problem by using as stochastic oracle the monotone operator associated to the stochastic convex-concave function \eqref{eqn:monotone_op_SP}, over the set $\W={\cal X}\times{\cal Y}$. From here onwards, when we say that a certain SVI algorithm is applied to an SSP, we mean using the choices above for the operator and feasible set, respectively. 

We start by stating the convergence guarantees for the  single-pass \knseg\, method. This is obtained as a direct corollary of 
\cite[Thm.~1]{Juditsky:2011}, where we use an explicit bound on the oracle error with the variance of the Gaussian.

\begin{theorem}[\cite{Juditsky:2011}]\label{thm:dp_generalization_conv}
Consider a stochastic variational inequality \eqref{main_problem}, with operators $F_{\bbeta}$ 
in ${\cal M}^1(L,M)$. Let $\A$ be the \knseg~method (Algorithm \ref{alg:alg1}) where \textcolor{red}{$0<\gamma_t\leq 1/[\sqrt3 L]$} and $(\bB_t^1,\bB_t^2)_{t}$ are 
independent random variables from a product distribution $\bB_t^1,\bB_t^2\sim{\cal P}^{B_t}$, satisfies
\begin{equation*} 
\VIgap{{\cal A}}{F}
\leq \frac{K_0(T)}{\Gamma_T},
\end{equation*}
where $K_0(T)\triangleq\Big( D^2 +7\sum_{t\in[T]}\gamma_t^2[M^2/2+d\sigma_t^2]\Big)$, 
$\Gamma_T = \tsum_{t=1}^T\gamma_t$, ${\cal A}(\bS)$ is the output of Algorithm \ref{alg:alg1} on the dataset $\bS=\bigcup_t \bB_t \sim{\cal P}^n$ and expectation in the left hand side is taken over the dataset draws, random sample batch choices, as well as noise $\bg^t_1, \bg^t_2$. 

On the other hand, $\A$ applied to a stochastic \eqref{spp_problem} problem attains saddle point gap 
\begin{equation*} 
\SPgap{{\cal A}}{f}
\leq \frac{K_0(T)}{\Gamma_T}.
\end{equation*}
\end{theorem}

\subsection{Differential privacy analysis of \knseg~ method}

We now proceed to establish the privacy guarantees for the single-pass variant of Algorithm \ref{alg:alg1}. This is 
a direct consequence of Propositions \ref{propos:Gauss_mech} and \ref{propos:parallel_comp}, and the fact that 
each operator evaluation has sensitivity bounded by $2M/B_t$. 
\begin{proposition} \label{prop:privacy_NSEG_step}
Algorithm \ref{alg:alg1} with batch sizes $(B_t)_{t\in[T]}$ and variance $\sigma_t^2=\frac{8M^2}{B_{t}^2}\frac{\ln(1/\privb)}{\priva^2}$ 
is $(\priva, \privb)$-differentially private.
\end{proposition}


We now apply the previous results to obtain population risk bounds for DP-SVI by the \knseg\, method.
\begin{corollary}
Algorithm \ref{alg:alg1} with disjoint batches of size $B_t=B=\min\{\sqrt{d(\ln(1/\privb)}/\priva, n\}$, constant stepsize $\gamma_t\equiv\gamma=D/\big[M\sqrt{7T\big(1+\frac{8d}{B^2}\frac{\ln(1/\privb)}{\priva^2}\big)}\big]$ and variance 
$\sigma_t^2=\sigma^2=\frac{8M^2}{B^2}\frac{\ln(1/\privb)}{\priva^2}$ 
is $(\priva, \privb)$-differentially private and achieves $\VIgap{{\cal A}}{F}$ (for SVI) or $\SPgap{\A}{f}$ (for SSP) of
$$ O\Big(MD \max\Big\{ \frac{[d \ln(1/\privb)]^{1/4}}{\sqrt{n\priva}}, \frac{\sqrt{d\ln(1/\privb)}}{n\priva} \Big\}\Big). $$
\end{corollary}

\begin{remark}
Notice that in the corollary above, the gap is nontrivial iff $\sqrt{d\ln(1/\privb)}/[n\priva]<1$, which means that
the left hand side attains the max on the range where the gap is nontrivial.
\end{remark}

\begin{proof}
Consider a SVI or SSP problem. 
Let us recall that by Theorem \ref{thm:dp_generalization_conv}, Algorithm \ref{alg:alg1} achieves expected gap
$$ \frac{D^2}{\gamma T}+7M^2\gamma \Big(1+\frac{8d}{B^2}\frac{\ln(1/\privb)}{\priva^2} \Big).$$
Choosing $\gamma=D/\big[M\sqrt{7T\big(1+\frac{8d}{B^2}\frac{\ln(1/\privb)}{\priva^2}\big)}\big]$, we obtain 
an expected gap
$$ \frac{2\sqrt{7}MD}{\sqrt{T}}\Big(1+\frac{\sqrt{8d}}{B\priva}\sqrt{\ln(1/\privb)}\Big)
=\frac{2\sqrt{14}MD\sqrt{B}}{\sqrt{n}}+\frac{8\sqrt{7}MD\sqrt{d}}{\sqrt{nB}\priva}\sqrt{\ln(1/\privb)},$$
where we used that for a single-pass algorithm, $n=2TB$ (this choice of $T$ exhausts the data when disjoint batches are chosen). 

Recalling that $B=\min\{\sqrt{d\ln(1/\privb)}/\priva, n\}$. Then the expected gap is bounded by
$$ O\Big(MD \max\Big\{ \frac{[d\ln(1/\privb)]^{1/4}}{\sqrt{n\priva}}, \frac{\sqrt{d\ln(1/\privb)}}{n\priva} \Big\}\Big). $$
\end{proof}

We observe that excess risk bounds of the same order for DP-SCO based on noisy SGD and the uniform stability of differential privacy have 
been established \cite{Bassily:2014}. 
 Improving these bounds in DP-SCO required substantial efforts, which was only achieved recently
\cite{Bassily:2019,Feldman:2020,BFGT:2020}. Furthermore, to the best of our knowledge, the upper bounds on the risk above are the first of their type for DP-SVI and DP-SSP, respectively. To improve upon them, we will follow the approach of \cite{BFGT:2020}, based on a multi-pass 
empirical error convergence, combined with \added{weak gap} generalization bounds based on uniform stability.

%% file: sec_NSEG_stability_priv.tex
\section{Stability of NSEG and Optimal Risk for DP-SVI and DP-SSP}\label{sec:stab_nseg}

The bounds established for DP-SVI are potentially suboptimal, and many of the past approaches used to 
attain optimal rates for DP-SCO, such as privacy amplification by iteration, phased regularization, etc. appear to 
encounter substantial barriers for their application to DP-SVI. In order to resolve this gap, we show
that for both DP-SVI and DP-SSP we can indeed obtain optimal rates,
which match those of DP-SCO. In order to achieve this we develop a {\em multi-pass} variant of the \knseg\, method, \added{which enjoys generalization performance due to its stability.}

\subsection{Stability of NSEG~method}

\label{subsec:stab_NSEG}

To analyze the stability of \knseg\, it is useful to interpret the extragradient method as an approximation of the proximal point 
algorithm. This connection has been established at least since \cite{Nemirovski:2004}. Given a monotone 
and $1$-Lipschitz operator $G:\bE\mapsto\bE$, 
we define the $s$-extragradient
operator inductively as follows. First, $\T{0}{G}{\cdot}:\bE\mapsto\bE$ is defined as $\T{0}{G}{u}= \proj_{\W}(u)$. Then, for
$s\geq 0$
\begin{equation} \label{eqn:extrapolation}
 \T{s+1}{G}{u}= \proj_{\W}(u- G(\T{s}{G}{u})).
\end{equation}
Given such operator, the (deterministic) extragradient method \cite{Korpelevich:1976} 
corresponds to, 
starting from $u_0\in \W$, iterating
$$ u_{t+1} = \T{2}{\gamma F}{u_t} \qquad(\forall t\in[T-1]).$$

It is known that if $G$ is contractive, the recursion \eqref{eqn:extrapolation} leads to
a fixed point $\T{}{G}{u}$, satisfying
\begin{equation} \label{eqn:fp_extrapolation} 
\T{}{G}{u}= \proj_{\W}(u- G(\T{}{G}{u})). 
\end{equation}
It is also easy to see that $\T{}{G}{\cdot}:\bE\mapsto\W$ is nonexpansive.




\begin{proposition}[Near nonexpansiveness of the extragradient operator] \label{lem:near_nonexp_fxpt}
Let $F\in {\cal M}_{\W}^1(L,M)$ and $\W\subseteq \bE$ compact convex set with diameter $D>0$. Then, for all $s$
nonnegative integer, and $u,v\in\bE$,
\begin{equation} \label{eqn:fixed_pt_conv}
\| \T{s}{\gamma F}{u}- \T{}{\gamma F}{u}\| \leq 
(\gamma L)^s\|\T{0}{\gamma F}{u}-\T{}{\gamma F}{u} \|
\end{equation}
and
\begin{equation} \label{eqn:near_nonexp_fixed_pt}
\| \T{s}{\gamma F}{u}- \T{s}{\gamma F}{v}\| \leq \|u-v\|+2D(\gamma L)^s.
\end{equation}
\end{proposition}
\begin{proof}
The first part, eqn.~\eqref{eqn:fixed_pt_conv}, is proved by induction on $s$. The result clearly holds for $s=0$, and if $s\geq 1$, we use \eqref{eqn:extrapolation} and \eqref{eqn:fp_extrapolation} to obtain
\begin{eqnarray*}
\| \T{s}{\gamma F}{u} - \T{}{\gamma F}{u} \| &=& \|\proj_{\W} (u-\gamma F(\T{s-1}{\gamma F}{u})- \proj_{\W}(u-\gamma F(\T{}{\gamma F}{u}))\| \\
 &\leq&  \gamma\| F(\T{s-1}{\gamma F}{u})-F( \T{}{\gamma F}{u} )\| \,\,\leq\,\,  \gamma L\| \T{s-1}{\gamma F}{u}- \T{}{\gamma F}{u} \| \\
 &\leq& (\gamma L)^s\| \T{0}{\gamma F}{u}-\T{}{\gamma F}{u}\|,
\end{eqnarray*}
where in the first inequality we used the nonexpansiveness of
the projection operator, next we used the $L$-Lipschitzness of $F$, and finally we used the inductive hypothesis to conclude. 

The second part, eqn.~\eqref{eqn:near_nonexp_fixed_pt}, is a direct consequence of \eqref{eqn:fixed_pt_conv}, the triangle inequality, and that $\T{0}{\gamma F}{u}$, $\T{}{\gamma F}{u}$, $\T{0}{\gamma F}{v}$, $\T{}{\gamma F}{v}\in \W$.
\end{proof}

The next lemma shows an expansion upper bound for extragradient iterations. This type 
of bound will be later used to establish the uniform argument stability of the NSEG algorithm.
\begin{lemma}[Expansion of the extragradient iteration] \label{lem:expansion_EG}
Let $F_1,F_2:\bE\mapsto\bE$ monotone $L$-Lipschitz operators, 
and $0\leq \gamma< 1/L$. 
Let $u,v\in \W$, and $w,z,u^{\prime},v^{\prime}\in \W$ such that 
\begin{eqnarray*}
w &=& \proj_{\W}(u-\gamma F_1(u)) 
\qquad\qquad\,
z \,\,\,=\,\,  \proj_{\W}(v-\gamma F_1(v))  \\
u^{\prime} &=&  \proj_{\W}(u-\gamma F_2(w)) 
\qquad\qquad v^{\prime} \,\,=\,\,  \proj_{\W}(v-\gamma F_2(z)) .
\end{eqnarray*}
Then, 
\begin{eqnarray}
\|w-z\| &\leq& \|u-v\|+2 L D\gamma,  \label{eqn:exp_bd_delta}\\
\|u^{\prime}-v^{\prime}\| &\leq& \|u-v\|+(\M_1+\M_2+2LD)L \gamma^2, \label{eqn:exp_bd_F1neqF2}
\end{eqnarray}
where $\M_1\triangleq \|F_1(\T{}{\gamma F_1}{u})-F_2(\T{}{\gamma F_2}{u})\|$ and $\M_2\triangleq \|F_1(\T{}{\gamma F_1}{v})-F_2(\T{}{\gamma F_2}{v})\|$.
\end{lemma}
\begin{proof}
By definition of $w$ and $z$, we have,
\begin{eqnarray*}
\|w-z\| &= & \| \proj_{\W}(u-\gamma F_1(u)) -\proj_{\W}(v-\gamma F_1(v))  \| \\ 
&\leq &\| \T{}{\gamma F_1}{u}- \T{}{\gamma F_1}{v}\|+\| \proj_{\W}(u-\gamma F_1(u))  - \T{}{\gamma F_1}{u}\|\\
&&+\| \proj_{\W}(v-\gamma F_1(v))  - \T{}{\gamma F_1}{v} \| \\
&\leq&\|u-v\|+ \gamma L [\| \T{0}{\gamma F_1}{u} - \T{}{\gamma F_1}{u} \|+\| \T{0}{\gamma F_1}{v}-T{}{\gamma F_1}{v}\|],
\end{eqnarray*}
where we used the nonexpansiveness of the operator
$\T{}{\gamma F_1}{\cdot }$ and Proposition \ref{lem:near_nonexp_fxpt}. Moreover, since $u,v,\T{0}{\gamma F_1}{u}, \T{0}{\gamma F_1}{v} \in\W$, we have $\|w-z\|\leq \|u-v\|+2LD\gamma$, proving \eqref{eqn:exp_bd_delta}. 

Next, to prove \eqref{eqn:exp_bd_F1neqF2}, we proceed as follows:
\begin{eqnarray*}
\|u^{\prime}-v^{\prime}\| 
&= & \|\proj_{\W}(u-\gamma F_2(w))-\proj_{\W}(v-\gamma F_2(z))\|\\
&\leq& \| \T{}{\gamma F_2}{u} - \T{}{\gamma F_2}{v}\|+
 \|\proj_{\W}(u-\gamma F_2(w))-\proj_{\W}(u-\gamma F_2(\T{}{\gamma F_2}{u}))\| \\
 && + \|\proj_{\W}(v-\gamma F_2(z))-\proj_{\W}(v-\gamma F_2(\T{}{\gamma F_2}{v}))\|\\
&\leq& \|u-v\|+\gamma L\|w- \T{}{\gamma F_2}{u}\|
+\gamma L \|z- \T{}{\gamma F_2}{v}\|.
\end{eqnarray*}
Using again Proposition \ref{lem:near_nonexp_fxpt}, we have that
\begin{eqnarray*}
 \|w- \T{}{\gamma F_2}{u} \|
 &=& \| \T{1}{\gamma F_1}{u}- \T{}{\gamma F_2}{u} \| \\ &\leq& \| \T{1}{\gamma F_1}{u}-\T{}{\gamma F_1}{u} \| + \| \T{}{\gamma F_1}{u} - \T{}{\gamma F_2}{u} \|  \\
&\leq&  L D\gamma+\|\proj_{\W}\big(u-\gamma F_1( \T{}{\gamma F_1}{u})\big)-\proj_{\W}\big(u-\gamma F_2( \T{}{\gamma F_2}{u})\big)\| \\
&\leq&  L D\gamma+\gamma \|F_1(\T{}{\gamma F_1}{u})-F_2(\T{}{\gamma F_2}{u})\|\\
&\leq&  L D\gamma+\M_1\gamma.
\end{eqnarray*}
An analog bound can be obtained for $\|z-\T{}{\gamma F_2}{v}\|$:
$$ \|z-\T{}{\gamma F_2}{v}\|\leq LD\gamma+\M_2\gamma, $$ 
concluding the claimed bound \eqref{eqn:exp_bd_F1neqF2}:
$$ \|u^{\prime}-v^{\prime}\|\leq \|u-v\|+ L[\M_1+\M_2+ 2LD]\gamma^2.$$
\end{proof}

The Expansion Lemma above allows us to bound how much would two trajectories of the \knseg\, method may deviate, given two pairs of sequences of operators $F_{1,t},F_{2,t}$ and $F_{1,t}^{\prime},F_{2,t}^{\prime}$. 
The bounds we will obtain from this analysis will give us direct bounds on the UAS 
for the NSEG\, method.

\begin{lemma} \label{lem:stab_EG}
Let $F_{1,t},F_{2,t}$ and $F_{1,t}^{\prime},F_{2,t}^{\prime}$ be $L$-Lipschitz operators,
and $0\leq \gamma_t< 1/L$ for all $t\in[T]$. Let 
$\{(u_t,w_t)\}_{t\in[T]}$ and $\{(v_t,z_t)\}_{t\in[T]}$ be the sequences resulting from Algorithm \ref{alg:alg1}, with operators $\{(F_{1,t},F_{2,t})\}_{t\in[T]}$ and $\{(F_{1,t}^{\prime},F_{2,t}^{\prime})\}_{t\in[T]}$, respectively; and starting from $u^0=v^0$. Let 
\begin{eqnarray*}
\Delta_{1,t} &\triangleq& \sup_{u\in\W}\|F_{1,t}(u)-F_{1,t}^{\prime}(u)\|, \\
\Delta_{2,t}  &\triangleq&\sup_{u\in\W}\|F_{2,t}(u)-F_{2,t}^{\prime}(u)\|,\\
\M_{1,t}       &\triangleq& \|F_{1,t}(\T{}{\gamma F_{1,t}}{u_{t-1}})-F_{2,t}(\T{}{\gamma F_{2,t}}{u_{t-1}})\|, \mbox{ and } \\
\M_{2,t}       &\triangleq& \|F_{1,t}(\T{}{\gamma F_{1,t}}{v_{t-1}})-F_{2,t}(\T{}{\gamma F_{2,t}}{v_{t-1}})\|;
\end{eqnarray*}
then, for al $t=0,\ldots,T$,
\begin{eqnarray}
 \nu_t \triangleq \|u_t-v_t\| &\leq& \sum_{s=1}^{t}\Big( [\M_{1,t}+\M_{2,t}+2LD]L\gamma_s^2+L\Delta_{1,s}\gamma_s^2+\Delta_{2,s}\gamma_s\Big)
 \label{eqn:bd_nu_F1neqF2}\\
\delta_t \triangleq  \|w_t-z_t\| &\leq&  \sum_{s=1}^{t-1}\Big( [\M_{1,t}+\M_{2,t}+2LD]L\gamma_s^2+L\Delta_{1,s}\gamma_s^2+\Delta_{2,s}\gamma_s\Big)
+\Delta_{1,t}\gamma_t+2LD\gamma_t.
\label{eqn:bd_delta_F1neqF2}
\end{eqnarray}
\end{lemma}
\begin{proof}
Clearly, $\nu_0=0$. Let us now derive a recurrence for both $\nu_t$ and $\delta_t$. 
\begin{eqnarray*}
\delta_{t} &=& \| \T{1}{\gamma_t F_{1,t}}{u_{t-1}} - \T{1}{\gamma_t F_{1,t}^{\prime}}{v_{t-1}} \|  \\
&\leq&  \| \T{1}{\gamma_t F_{1,t}}{u_{t-1}} - \T{1}{\gamma_t F_{1,t}}{v_{t-1}} \|+
\| \T{1}{\gamma_t F_{1,t}}{v_{t-1}} - \T{1}{\gamma_t F_{1,t}^{\prime}}{v_{t-1}} \| \\
&\leq& \|u_{t-1}-v_{t-1}\|+2LD\gamma_t + \| \T{1}{\gamma_t F_{1,t}}{v_{t-1}} - \T{1}{\gamma_t F_{1,t}^{\prime}}{v_{t-1}}\|,
\end{eqnarray*}
where in the last inequality we used inequality \eqref{eqn:exp_bd_delta}.
Let us bound now the rightmost term above,
\begin{eqnarray}
\| \T{1}{\gamma_t F_{1,t}}{v_{t-1}} - \T{1}{\gamma_t F_{1,t}^{\prime}}{v_{t-1}} \|
 &=& \|\proj_{\W}(v_{t-1}-\gamma_t F_{1,t}(v_{t-1}))-\proj_{\W}(v_{t-1}-\gamma_t F_{1,t}^{\prime}(v_{t-1}))\|  \notag\\
 &\leq&\gamma_t \| F_{1,t}(v_{t-1})- F_{1,t}^{\prime}(v_{t-1})\| \notag \\
 &\leq & \Delta_{1,t}\gamma_t.
\label{eqn:rec_bd_stab}
\end{eqnarray}
We conclude that 
\begin{equation} \label{eqn:bd_delta}
\delta_{t} \leq \nu_{t-1} + 2LD\gamma_t + \Delta_{1,t}\gamma_t.
\end{equation}


Now,
\begin{align*}
\nu_{t} =& \|u_{t}-v_{t}\| \leq \|\proj_{\W}(u_{t-1}-\gamma_t F_{2,t}(w_t))-\proj_{\W}(v_{t-1}-\gamma_t F_{2,t}^{\prime}(z_t))\| \notag \\
=& \|\proj_{\W}(u_{t-1}-\gamma_t F_{2,t}(w_{t}))-\proj_{\W}(v_{t-1}-\gamma_t F_{2,t}(z_t))\|\notag \\
&+\|\proj_{\W}(v_{t-1}-\gamma_t F_{2,t}(z_t))-\proj_{\W}(v_{t-1}-\gamma_t F_{2,t}^{\prime}(z_t))\| \notag\\
\leq& \|\proj_{\W}(u_{t-1}-\gamma_t F_{2,t}( \T{1}{\gamma_t F_{1,t}}{u_{t-1}}))-\proj_{\W}(v_{t-1}-\gamma_t F_{2,t}(\T{1}{\gamma_t F_{1,t}^{\prime}}{v_{t-1}}))\| \notag\\
&+\gamma_t \|F_{2,t}(z_t)-F_{2,t}^{\prime}(z_t)\| \notag \\
\leq&\|\proj_{\W}(u_{t-1}-\gamma_t F_{2,t}(\T{1}{\gamma_t F_{1,t}}{u_{t-1}}))-\proj_{\W}(v_{t-1}-\gamma_t F_{2,t}(\T{1}{\gamma_t F_{1,t}}{v_{t-1}}))\| \notag \\
&+\|\proj_{\W}(v_{t-1}-\gamma_t F_{2,t}(\T{1}{\gamma_t F_{1,t}}{v_{t-1}}))-\proj_{\W}(v_{t-1}-\gamma_t F_{2,t}(\T{1}{\gamma_t F_{1,t}^{\prime}}{v_{t-1}}))\| +\gamma_t \Delta_{2,t}  \notag \\
\mleq{(i)}& \|u_{t-1}-v_{t-1}\|+[\M_{1,t}+\M_{2,t}+2LD]L\gamma_t^2
+\gamma_tL\| \T{1}{\gamma_t F_{1,t}}{v_{t-1}} - \T{1}{\gamma_t F_{1,t}^{\prime}}{v_{t-1}}\|+\gamma_t \Delta_{2,t} 
\\
\mleq{(ii)}& \nu_{t-1} + [\M_{1,t}+\M_{2,t}+2LD]L\gamma_t^2+L\Delta_{1,t}\gamma_t^2+\Delta_{2,t}\gamma_t, 
\end{align*}
where in inequality (i), 
we used Lemma \ref{lem:expansion_EG} (more precisely, inequality \eqref{eqn:exp_bd_F1neqF2}), 
and in inequality (ii), 
we used \eqref{eqn:rec_bd_stab}. Unraveling the above recursion, we get that for all $t\in[T]$,
$$ \nu_t \leq \sum_{s=1}^{t}\Big( [\M_{1,t}+\M_{2,t}+2LD]L\gamma_s^2+L\Delta_{1,s}\gamma_s^2+\Delta_{2,s}\gamma_s\Big).$$
Finally, we combine the bound above with \eqref{eqn:bd_delta}, to conclude that for all $t\in[T]$:
$$ \delta_t \leq \sum_{s=1}^{t-1}\Big( [\M_{1,t}+\M_{2,t}+2LD]L\gamma_s^2+L\Delta_{1,s}\gamma_s^2+\Delta_{2,s}\gamma_s\Big)
+\Delta_{1,t}\gamma_t+2LD\gamma_t.$$
\end{proof}

The next theorem provides in-expectation and high probability upper bounds for the \knseg\, method. Despite the fact that we will not particularly apply the latter bounds, we believe these may be of independent interest.

\begin{theorem} \label{thm:stab_k_EG}
The \knseg\, method (Algorithm \ref{alg:alg1}) for closed and convex domain $\W\subseteq \bE$ 
with diameter $D$, operators in ${\cal M}_{\W}^1(L, M)$ and stepsizes $0<\gamma_t\leq 1/L$, satisfies the 
following uniform argument stability bounds:
\begin{enumerate}
\item {\em Batch method $\AbatchEG$:} \added{Where given dataset $\bS$, $F_{1,t}=F_{\bS}+\bg_1^t$, and $F_{2,t}=F_{\bS}+\bg_2^t$.} In expectation,
$$\sup_{\bS\simeq \bS^{\prime}}\EE_{\AbatchEG}[\delta_{\AbatchEG}(\bS,\bS^{\prime})]\leq 
 \sum_{t=0}^{T-1}\Big( [4M+2LD+4\sqrt{d}\sigma]L\gamma_t^2+\frac{2ML}{n}\gamma_t^2+\frac{2M}{n}\gamma_t\Big)
+ \frac{1}{T}\sum_{t=1}^{T}\big(\frac{2ML}{n}+2LD\big)\gamma_t,$$
And for constant stepsize $\gamma_t\equiv\gamma$, 
there exists a universal constant $K>0$, such that for any $0<\theta<1$, with probability $1-\theta$:
\begin{equation} \label{eqn:h_p_bd_batch_NSEG}
\sup_{\bS\simeq \bS^{\prime}}\delta_{\AbatchEG}(\bS,\bS^{\prime})
\leq 4[T\sqrt{d}\sigma+\sigma\sqrt{Kd\ln(1/\theta)}]L\gamma^2+
 [4M+2LD]LT\gamma^2+\frac{2ML}{n}T\gamma^2+\frac{2M}{n}T\gamma
+ \big(\frac{2ML}{n}+2LD\big)\gamma.
\end{equation}
\item {\em Sampled with replacement $\AreplEG$:} \added{Where given dataset $\bS$, $F_{1,t}=F_{\bbeta_{i(1,t)}}+\bg_1^t$, and $F_{2,t}=F_{\bbeta_{i(2,t)}}+\bg_2^t$, for $i(1,t),i(2,t)\sim\mbox{Unif}([n])$, independently.} In expectation,
\begin{equation} \label{eqn:exp_UAS_stoch_EG}
\sup_{\bS\simeq \bS^{\prime}}\EE[\delta_{\AreplEG}(\bS,\bS^{\prime})] 
\,\,\leq\,\, \sum_{t=0}^{T-1}\Big( [4M+2LD+4\sqrt{d}\sigma]L\gamma_t^2+\frac{2 ML}{n}\gamma_t^2+\frac{2 M}{n}\gamma_t\Big)
+ \frac{1}{T}\sum_{t=1}^{T}\big(\frac{2ML}{n}+2LD\big)\gamma_t.
\end{equation}
And for constant stepsize $\gamma_t\equiv\gamma$, there exists a universal constant $K>0$, such that
for any $0<\theta<1/[2n]$, with probability $1-\theta$:
\begin{multline}  \label{eqn:HP_UAS_stoch_EG} 
\sup_{\bS\simeq \bS^{\prime}}\delta_{\AreplEG}(\bS,\bS^{\prime})\leq  4[T\sqrt{d}\sigma+\sigma\sqrt{Kd\ln(2/\theta)}]L\gamma^2+ [4M+2LD]LT\gamma^2 +2LD\gamma \\
+\big(1+3\log\big(\frac{2n}{\theta}\big)\big)\frac{2MT}{n}(L\gamma^2+\gamma/T+\gamma).
\end{multline}
\end{enumerate}
\end{theorem}

\begin{proof} Let $\bS\simeq \bS^{\prime}$. Then 
\begin{enumerate}
\item {\bf Batch method.} Notice that for the batch case $F_{1,t}=F_{\bS}+\bg_1^t$, and
$F_{1,t}^{\prime}=F_{\bS^{\prime}}+\bg_1^t$; and $F_{2,t}=F_{\bS}+\bg_2^t$, and 
$F_{2,t}^{\prime}=F_{\bS^{\prime}}+\bg_2^t$. 
Then, it is easy to see that $\Delta_{1,t}\leq2M/n$ and $\Delta_{2,t}\leq2M/n$. On the other hand,
since the operators are $M$ bounded and since noise addition is Gaussian
\begin{eqnarray}
\EE[\M_{1,t}] &=& \EE[\|F_{1,t}(\T{}{\gamma F_{1,t}}{u_{t-1}})-F_{2,t}(\T{}{\gamma F_{2,t}}{u_{t-1}})\|] \nonumber \\
&\leq&  \EE[\|F_{\bB_t^1}(\T{}{\gamma F_{1,t}}{u_{t-1}})+\bg_1^t\|]+\EE[\|F_{\bB_t^2}(\T{}{\gamma F_{2,t}}{u_{t-1}})+\bg_2^t\|] \nonumber\\
&\leq& 2M+\EE[\|\bg_1^t\|+\|\bg_2^t\|] \leq 2[M+\sqrt{d}\sigma], \label{eqn:bd_noise_stab}
\end{eqnarray}
and an analog bound holds for $\EE[\M_{2,t}]$. Hence, by Lemma \ref{lem:stab_EG}:
$$\EE_{\AbatchEG}[\delta_{\AbatchEG}(S,S^{\prime})]\leq 
 \sum_{t=0}^{T-1}\Big( [4M+2LD+4\sqrt{d}\sigma]L\gamma_t^2+\frac{2ML}{n}\gamma_t^2+\frac{2M}{n}\gamma_t\Big)
+ \frac{1}{T}\sum_{t=1}^{T}\big(\frac{2ML}{n}+2LD\big)\gamma_t,$$
which proves the claimed bound. 

For the high probability bound, we use that the norm of a Gaussian vector is $Kd\sigma^2$-subgaussian,
for a universal constant $K>0$ (see, e.g.~\cite[Thm.~3.1.1]{Vershynin:2018}), and therefore 
$\EE[\exp\{\lambda(\|\bg_i^t\|-\sigma\sqrt{d})\}]\leq \exp\{Kd\sigma^2\lambda^2\}$; hence by the Chernoff-Cr\'amer bound,
for any $\alpha>0$
\begin{eqnarray*}
\PP\Big[\sum_{t\in[T]} \big(\|\bg_1^t\|+\|\bg_2^t\|\big) >(2+\alpha)T\sqrt{d}\sigma)\Big]
&\leq& \exp\{-\lambda\alpha T\sqrt{d}\sigma\}\Big( \exp\{2Kd\sigma^2\lambda^2\} \Big)^T\\
&=& \exp\{T(2Kd\sigma^2\lambda^2-\alpha\sqrt{d}\sigma\lambda)\}.
\end{eqnarray*}
Choosing $\lambda=\alpha/[4K\sqrt{d}\sigma]$ and $\alpha=\frac{2\sqrt{K}}{T}\sqrt{\ln(1/\theta)}$, we get 
\begin{equation} \label{eqn:concentration_gaussians}
\PP\Big[\sum_{t\in[T]} \big(\|\bg_1^t\|+\|\bg_2^t\|\big) > 2T\sqrt{d}\sigma+2\sigma\sqrt{Kd\ln(1/\theta)}\Big]\leq\theta.
\end{equation}
This guarantee, together with the rest of the terms appearing in our previous stability bound (which hold w.p.~1) 
proves \eqref{eqn:h_p_bd_batch_NSEG}.
\item {\bf Sampled with replacement.} 
Let $i\in[n]$ be the coordinate where $\bS$ and $\bS^{\prime}$ may differ.
Let $i(1,t),i(2,t)\sim \mbox{Unif}([n])$ i.i.d., for $t\in [T]$. Now we apply Lemma \ref{lem:stab_EG} with 
$F_{1,t}=F_{\bbeta_{i(1,t)}}+\bg_1^t$, and $F_{1,t}^{\prime}=F_{\bbeta_{i(1,t)}^{\prime}}+\bg_1^t$; 
and $F_{2,t}=F_{\bbeta_{i(2,t)}}+\bg_2^t$, and $F_{2,t}^{\prime}=F_{\bbeta_{i(2,t)}^{\prime}}+\bg_2^t$. 
Hence we have that $(\Delta_{1,t})_{t\in[T]}$ and $(\Delta_{2,t})_{t\in[T]}$ are
sequences of independent r.v.~with expectation bounded by $2M/n$. Therefore,
by Lemma \ref{lem:stab_EG}  (eqn.~\eqref{eqn:bd_delta_F1neqF2}), and following the
steps that lead to inequality \eqref{eqn:bd_noise_stab}, we have:
\begin{eqnarray*}
\EE[\delta_{\cal A}(\bS,\bS^{\prime})] 
&\leq& \sum_{t=0}^{T-1}\Big( [4M+2LD+4\sqrt{d}\sigma]L\gamma_t^2+\frac{2ML}{n}\gamma_t^2+\frac{2M}{n}\gamma_t\Big)
+ \frac{1}{T}\sum_{t=1}^{T}\big(\frac{2ML}{n}+2LD\big)\gamma_t.
\end{eqnarray*}
Finally for the high-probability bound, note that for any realization of the algorithm randomness, we have
$$ \delta_{\cal A}(\bS,\bS^{\prime}) \leq \sum_{t=1}^T[4M+2(\|\bg_1^t\|+\|\bg_2^t\|)+2LD]L\gamma_t^2+\frac{2LD}{T}\sum_{t=1}^T\gamma_t
+L\sum_{t=1}^T\gamma_t^2\Delta_{1,t}+\frac1T\sum_{t=1}^T\Delta_{1,t}\gamma_t+\sum_{t=1}^T\Delta_{2,t}\gamma_t.$$
We additionally assume constant stepsize, $\gamma_t\equiv \gamma>0$.
Hence, we can resort on concentration of sums of Bernoulli random variables, which guarantees that
$$ \PP\Big[ \sum_{t=1}^T\Delta_{1,t} > (1+3\log(2/\theta))\frac{2MT}{n} \Big]
\leq \exp\{-\log(2/\theta)\}=\frac{\theta}{2}. $$
An analog bound can be established for $\Delta_{2,t}$, which together with bound \eqref{eqn:concentration_gaussians}
leads to
\begin{multline*} 
\PP_{\AreplEG}\Big[ \delta_{\AreplEG}(\bS,\bS^{\prime})> 4[T\sqrt{d}\sigma+\sigma\sqrt{Kd\ln(1/\theta)}]L\gamma^2+ [4M+2LD]LT\gamma^2 +2LD\gamma \\
+\big(1+3\log\big(\frac{2}{\theta}\big)\big)\frac{2MT}{n}(L\gamma^2+\gamma/T+\gamma) \Big]\leq 2\theta. 
\end{multline*}
Notice this bound only depends on our choice of $i$, and it is otherwise uniform over all $\bS\simeq \bS^{\prime}$. 
Finally, by a union bound on $i\in[n]$ (together with a renormalization of $\theta$), we have that
\begin{multline*} 
\PP_{\AreplEG}\Big[ \sup_{\bS\simeq \bS^{\prime}} \delta_{\AreplEG}(\bS,\bS^{\prime})> 4[T\sqrt{d}\sigma+\sigma\sqrt{Kd\ln(2/\theta)}]L\gamma^2+ [4M+2LD]LT\gamma^2 +2LD\gamma \\
+\big(1+3\log\big(\frac{4n}{\theta}\big)\big)\frac{2MT}{n}(L\gamma^2+\gamma/T+\gamma) \Big]\leq \theta. 
\end{multline*}
\end{enumerate}
\end{proof}

\subsection{Optimal risk for DP-SVI and DP-SSP by \knseg~method}

Now we use our stability and risk bounds for NSEG to derive optimal risk bounds for DP-SSP. For this, we use
the sampled with replacement variant, $\AreplEG$. 
\begin{equation} \label{eqn:NSEG_setting} 
F_{1,t}(\cdot)=F_{\boldsymbol{\beta}_{i(1,t)}}(\cdot)+\bg_1^t; \qquad
F_{2,t}(\cdot)=F_{\boldsymbol{\beta}_{i(2,t)}}(\cdot)+\bg_2^t.
\end{equation}
Using the moments accountant method (Theorem \ref{thm:dp_prox_multipass}) one can show the following.
\begin{proposition}[Privacy of sampled with replacement \knseg]
Algorithm \ref{alg:alg1} with operators given by eqn.~\eqref{eqn:NSEG_setting} and
$\sigma_t^2=8M^2\log(1/\privb)/\priva^2$, is $(\priva,\privb)$-differentially private.
\end{proposition}
\begin{theorem}[Excess risk of sampled with replacement NSEG]
Consider an instance of the  \eqref{main_problem} or \eqref{spp_problem} problem. 
Let $\A$ be the sampled with replacement variant \eqref{eqn:NSEG_setting} of \knseg~method (Algorithm \ref{alg:alg1}), with
$\gamma_t=\gamma=\min\{D/M,1/L\}/[n\max\{\sqrt n, \sqrt{d \ln(1/\privb)}/\priva\}]$, $\sigma_t^2=8M^2\log(1/\privb)/\priva^2$, $T=n^2$. 
Then, $\wVIgap{\A}{F}$ (for SVI) or $\wSPgap{\A}{f}$ (for SSP) are bounded by
$$
 O\Big( (MD+LD^2)\max\Big\{\frac{1}{\sqrt n}, \frac{\sqrt{d\ln(1/\privb)}}{n\priva}\Big\} +\frac{MLD}{n^{5/2}}\Big).
$$
\end{theorem}

\begin{remark}
Notice that assuming $n=\Omega(\min\{\sqrt{L},\sqrt{M/D}\})$ the bound of the Theorem simplifies to
$$
O\Big(  (MD+LD^2)\max\Big\{\frac{1}{\sqrt n}, \frac{\sqrt{d\ln(1/\privb)}}{n\priva}\Big\} \Big).
$$
This is quite a mild sample size requirement. In this range, when $M\geq LD$, our upper bound matches the excess risk bounds for DP-SCO \cite{Bassily:2019},
and \added{we will show these rates are indeed optimal for DP-SVI and DP-SSP as well}
\end{remark}

\begin{proof}
Given that our bounds for SVI and SSP are analogous, we proceed indistinctively for both problems.

First, let us bound the empirical accuracy of the method. By Theorem \ref{thm:dp_generalization_conv}, together with the fact
that sampling with replacement is an unbiased stochastic oracle for the empirical operator:
\begin{eqnarray*}
\EE_{\bS}\Big[\mbox{EmpGap}(\A,\bS)\Big] 
&\leq& \frac{1}{\gamma T}\Big( \frac{D^2}{2}+7M^2T\gamma^2(1+\frac{8d \log(1/\privb)}{\priva^2}) \Big)\\
&\leq& \frac{D^2}{2n}\max\{M/D,L\} \max\{\sqrt{n},\sqrt{d\ln(1/\privb)}/\varepsilon\}
+\frac{7M^2\min\{D/M,1/L\}}{n\max\{\sqrt{n},\sqrt{d\ln(1/\privb)}/\varepsilon\}} \frac{9d\ln(1/\privb)}{\varepsilon^2} \\
&=& O\Big( (MD+LD^2)\max\Big\{\frac{1}{\sqrt n}, \frac{\sqrt{d\ln(1/\privb)}}{n\priva}\Big\} \Big),
\end{eqnarray*}
where $\mbox{EmpGap}(\A,\bS)$ is $\EmpVIgap{\A}{F_{\bS}}$ or
$\EmpSPgap{\A}{f_{\bS}}$ for an SVI or SSP problem, respectively.

Next, by Theorem \ref{thm:stab_k_EG}, we have that ${\cal A}$ (or $x(\bS)$ and $y(\bS)$, for the SSP case)
are UAS with parameter
\begin{eqnarray*}
\delta&=& [4M+2LD+4\sqrt{d}\sigma]LT\gamma^2+\frac{2 ML}{n}T\gamma^2+\frac{2 M}{n}T\gamma
+ \big(\frac{2ML}{n}+2LD\big)\gamma\\
&\leq& \Big(\frac{4LD^2}{M}+2D\Big)\frac1n+\frac{8LD^2\sqrt{2d\ln(1/\eta)}}{M\varepsilon n}+\frac{2LD^2}{M}\frac{1}{n^{3/2}}+\frac{2D}{\sqrt n}+\frac{2LD}{n^{5/2}}+\frac{2D}{n^{3/2}}\\
&=& O\Big(\frac{1}{M}\cdot\Big(\frac{MD+LD^2}{n}+\frac{LD^2\sqrt{d\ln(1/\eta)}}{\varepsilon n} +\frac{MLD}{n^{5/2}}\Big) \Big).
\end{eqnarray*}
Hence, noting that empirical risk upper bounds weak empirical gap and using Proposition \ref{prop:weak_generalization_SSP} or Proposition \ref{prop:weak_generalization_SVI} (depending on whether
the problem is an SSP or SVI, respectively), we have that the risk is upper bounded by its empirical risk plus $M\delta$, where
$\delta$ is the UAS parameter of the algorithm; in particular,   is bounded by
\begin{eqnarray*}
\wVIgap{\A}{F} &\le& \EE_{\bS}[\EmpVIgap{\A}{F_{\bS}}]+M\delta\\
&=& O\Big( (MD+LD^2)\max\Big\{\frac{1}{\sqrt n}, \frac{\sqrt{d\ln(1/\privb)}}{n\priva}\Big\} +\frac{MLD}{n^{5/2}}\Big),
\end{eqnarray*}
Similar claims can be made $\wSPgap{\A}{f_\bS}$. Hence, we conclude the proof.
\end{proof}

%% file: sec_prox_convergence.tex
\section{The Noisy Inexact Stochastic Proximal Point Method} \label{sec:prox_DPSVI}

In the previous sections, we presented \knseg\, method with its single-pass and multipass variants and provided optimal risk guarantees for DP-SVI and DP-SSP problems in $O(n^2)$ stochastic operator evaluations. In the rest of the paper, our aim is to provide another algorithm that can achieve the optimal risk for both of these problems with much less computational effort. Towards that end, consider the following algorithm: 
\begin{algorithm}[H]
	\caption{Noisy Inexact Stochastic Proximal Point (NISPP) Method}
	\label{alg:alg0}
	\begin{algorithmic}[1]
		\STATE {\bf Input:} $w_0 \in \W$
		\FOR {$k = 0, 1, \dots, K$}		
		\STATE Sample a batch $\bB_{k+1} \subseteq \bS$.
		\STATE $u_{k+1} \gets \text{VI}_\nu(\W, F_{\bB_{k+1}}(\cdot)+ \lambda_k (\cdot-w_k))$.
		\STATE $w_{k+1} \gets u_{k+1} + \boldsymbol{\xi}_{k+1}$, where $\boldsymbol{\xi}_{k+1} \sim \mathcal{N}(0, \sigma_{k+1}^2 \mathbb{I}_d)$ 
		\ENDFOR
		\STATE $\wb{w}_K :=  (\tsum_{k =0}^K \gamma_kw_{k+1})/(\tsum_{k = 0}^K\gamma_k)$
		\STATE {\bf Output:} $\proj_{\W}(\wb{w}_K)$ 
	\end{algorithmic}
\end{algorithm}
In the above algorithm, we leave few things unspecified which will be stated later during convergence and privacy analysis. Here, we detail some key features of the above algorithm. In line 3, we sample a batch $\bB_{k+1}$ of size $B_{k+1}=|\bB_{k+1}|$. 
Similar to the \knseg, we will look at two different variants of NISPP method: single-pass and multi-pass. Depending on the type of the method, we will specify the sampling mechanism. 
In line 4 of Algorithm \ref{alg:alg0}, we have $u_{k+1}$ is a $\nu$-approximate strong VI solution of the mentioned VI problem for some $\nu \ge 0$, i.e., 
\begin{equation}
	\label{eq:int_rel_prox_1}
	\inprod{F_{\bB_{k+1}}(u_{k+1}) + \lambda_k (u_{k+1}-w_k)}{w-u_{k+1}} \ge -\nu \qquad (\forall w \in \W).
\end{equation}
Note that if $\nu = 0$ then this is an exact solution satisfying \eqref{main_problem} with operator $F$ replaced by $F(\cdot) + \lambda_k(\cdot- w_k)$. For $\nu>0$, we obtain that $u_{k+1}$ is an inexact solution satisfying solution criterion up to $\nu$ additive error. In line 5, we add a Gaussian noise to $u_{k+1}$ in order to preserve privacy. The resulting iterate $w_{k+1}$ can be potentially outside the set $\W$. Hence, in line 7, the ergodic average $\wb{w}_K$ can be outside $\W$. In order to preserve feasibility of the solution, we project $\wb{w}_K$ onto set $\W$ and output it as a solution in line 8. Projection of the average in line 8, as opposed to projection individual $w_{k+1}$ in line 5 is crucial for convergence guarantee of Algorithm \ref{alg:alg0}.

In the rest of this section, we exclusively deal with the single-pass version of NISPP method, 
i.e., we assume that batches $\{\bB_{k+1}\}_{k=0,\ldots,K}$ are disjoint subsets of the dataset $\bS$.  
We start with the convergence guarantees of single-pass NISPP method.
In order to prove convergence, we show a useful bound on $\dist_\W(\wb{w}_K):= \min_{w\in \W}\gnorm{w-\wb{w}_K}{}{}$.
\begin{proposition}\label{prop:bound_for_projected_iterate}
	Let $\wb{u}_K := \frac{1}{\Gamma_K} \tsum_{k = 0}^K\gamma_ku_{k+1}$. Then,
	\begin{equation}\label{eq:int_rel_prox_23}
		\dist_\W(\wb{w}_K)^2\le \gnorm{\wb{u}_K - \wb{w}_K}{}{2} = \frac{1}{\Gamma_K^2}\gnorm{\tsum_{k = 0}^K\gamma_k\boldsymbol{\xi}_{k+1}}{}{2}.
	\end{equation}
	Moreover, we have 
	\begin{equation}\label{eq:int_rel_prox_24}
		\EE [\dist_\W(\wb{w}_K)^2] \le \frac{1}{\Gamma_K^2}\tsum_{k = 0}^K\gamma_k^2\EE\gnorm{\boldsymbol{\xi}_{k+1}}{}{2}
	\end{equation}
\end{proposition}
\begin{proof}
Note that $\wb{u}_K \in \W$. Hence, first relation in \eqref{eq:int_rel_prox_23} follows by definition of $\dist_\W(\cdot)$ function. Equality follows from definition of $\wb{u}_K$ and $\wb{w}_K$. 
To obtain \eqref{eq:int_rel_prox_24}, note that $\{\boldsymbol{\xi}_k\}_{k = 1}^{K+1}$ are i.i.d.~random variable with mean $0$. Expanding $\gnorm{\tsum_{k = 0}^K \gamma_k\boldsymbol{\xi}_{k+1}}{}{2}$, using linearity of expectation and noting $\EE[\boldsymbol{\xi}_i^T\boldsymbol{\xi}_j] = 0$ for all $i \ne j$, we conclude \eqref{eq:int_rel_prox_24}. Hence, we conclude the proof.
\end{proof}
We prove the following convergence rate result for Algorithm \ref{alg:alg0} for the risk of SVI/SSP problem. In particular, we assume that the algorithm performs a single-pass over the dataset $\bS \sim \Pcal^n$ containing $n$ i.i.d. datapoints. 

\begin{theorem}\label{thm:conv_prox_vi}
	Consider the stochastic \eqref{main_problem} problem with operators $F_{\bbeta} \in \M^1(L, M)$. Let $\A$ be the single-pass NISPP method (Algorithm \ref{alg:alg0}) where sequence $\{\gamma_k\}_{k \ge 0}, \{\lambda_k\}_{k \ge0}$ satisfy 
	\begin{equation}\label{eq:step_size_condition}
		\gamma_k \lambda_k = \gamma_0\lambda_0
	\end{equation} for all $k \ge 0$. Moreover, $\bB_{k+1}$ are independent samples from a product distribution $\bB_{k+1} \sim \Pcal^{B_{k+1}}$ and $\bB_{k+1} \subset \bS$. Then, we have
\begin{align}
	\VIgap{\A}{F} \le \nu + \frac{Z_0(K)}{\Gamma_K}  + M \sqrt{\frac{1}{\Gamma_K^2}\tsum_{k = 0}^K\gamma_k^2\sigma_{k+1}^2d}, \label{eq:conv_rate_svi_risk}
\end{align}
	where, $Z_0(K) := \frac{3\gamma_0\lambda_0}{2}D^2 +  \frac{4M^2 + 3L^2D^2}{\gamma_0\lambda_0}\tsum_{k = 0}^K\gamma_k^2 + \frac{5\gamma_0\lambda_0d}{2}\tsum_{k =1}^K\sigma_{k+1}^2$ and $\Gamma_K := \tsum_{k = 0}^K \gamma_k$. 
	
	Similarly, $\A$ applied to stochastic \eqref{spp_problem} problem achieves
	\begin{equation}\label{eq:conv_spp_risk}
		\SPgap{\A}{f} \le  \nu +  \frac{Z_0(K)}{\Gamma_K}  + M \sqrt{\frac{1}{\Gamma_K^2}\tsum_{k = 0}^K\gamma_k^2\sigma_{k+1}^2d}.
	\end{equation}
\end{theorem}
\begin{proof} Let $w\in\W$. Then
	\begingroup
	\allowdisplaybreaks
	\begin{align}
		\inprod{F(w)}{w_{k+1}-w} &= \inprod{F(w)}{u_{k+1}-w} + \inprod{F(w)}{w_{k+1}-u_{k+1}}.\label{eq:int_rel_prox_7}
	\end{align}
We will analyze each term above separately. First, note that 
\begin{align}
	\inprod{F(w)}{w_{k+1}-u_{k+1}} &\le \frac{1}{2\lambda_k}\gnorm{F(w)}{}{2} + \frac{\lambda_k}{2}\gnorm{\boldsymbol{\xi}_{k+1}}{}{2} \le \frac{M^2}{2\lambda_k} + \frac{\lambda_k}{2}\gnorm{\boldsymbol{\xi}_{k+1}}{}{2}.\label{eq:int_rel_prox_6}
\end{align}
Note that 
\begin{align}
	\inprod{F(w)}{u_{k+1}-w} &\le \inprod{F(u_{k+1})}{u_{k+1}-w}\nonumber\\
	&= \inprod{F_{\bB_{k+1}}(u_{k+1})}{u_{k+1}-w} + \inprod{F(u_{k+1}) - F_{\bB_{k+1}}(u_{k+1})}{u_{k+1}-w}\nonumber\\
	&\le \lambda_k\inprod{u_{k+1}-w_k}{w - u_{k+1}} + \nu + \inprod{F(u_{k+1}) - F_{\bB_{k+1}}(u_{k+1})}{u_{k+1}-w}\nonumber\\
	&=\frac{\lambda_k}{2} \bracket{\gnorm{w-w_k}{}{2} - \gnorm{w-u_{k+1}}{}{2} - \gnorm{u_{k+1} - w_k}{}{2}} + \nu \nonumber\\
	&\qquad+ \inprod{F(u_{k+1}) - F_{\bB_{k+1}}(u_{k+1})}{u_{k+1}-w}\label{eq:int_rel_prox_2},
\end{align}
where first inequality follows from monotonicity and second inequality follows from \eqref{eq:int_rel_prox_1}.
Now note that 
\begin{align*}
	&\inprod{F(u_{k+1}) - F_{\bB_{k+1}}(u_{k+1})}{u_{k+1}-w} \\
	&\quad= \inprod{F(u_{k+1}) - F_{\bB_{k+1}}(u_{k+1})- [F(w_k) - F_{\bB_{k+1}}(w_{k})]}{u_{k+1}-w}\\
	&\quad\quad + \inprod{F(w_k) - F_{\bB_{k+1}}(w_{k})}{u_{k+1} - w_k} + 
	\inprod{F(w_k) - F_{\bB_{k+1}}(w_{k})}{w_k-w}\\
	&\quad \le \frac{\lambda_k}{6L^2}\gnorm{F(u_{k+1}) - F(w_k)}{}{2} + \frac{3L^2}{2\lambda_k}\gnorm{u_{k+1}-w}{}{2}\\
	&\quad\quad + \frac{\lambda_k}{6L^2}\gnorm{F_{\bB_{k+1}}(u_{k+1}) - F_{\bB_{k+1}}(w_{k})}{}{2} + \frac{3L^2}{2\lambda_k}\gnorm{u_{k+1}-w}{}{2}\\
	&\quad\quad+  \frac{3}{2\lambda_k}\gnorm{F(w_k) - F_{\bB_{k+1}}(w_{k})}{}{2} +  \frac{\lambda_k}{6} \gnorm{u_{k+1}- w_k}{}{2}+ 
	\inprod{F(w_k) - F_{\bB_{k+1}}(w_{k})}{w_k-w}\\
	&\quad \le \frac{\lambda_k}{2}\gnorm{u_{k+1}- w_k}{}{2} + \frac{3L^2}{\lambda_k}\gnorm{u_{k+1}-w}{}{2} + \frac{3}{2\lambda_k}\gnorm{F(w_k) - F_{\bB_{k+1}}(w_{k})}{}{2} \\
	&\quad\quad+ 
	\inprod{F(w_k) - F_{\bB_{k+1}}(w_{k})}{w_k-w},
\end{align*}
where last inequality follows from $L$-Lipschitz continuity of $F$ and $F_{\bB_{k+1}}$. 
Noting that $\gnorm{u_{k+1}-w}{}{} \le D$ for all $w \in \W$ and using the above bound in \eqref{eq:int_rel_prox_2}, we have
\begin{align}
	\inprod{F(w)}{u_{k+1}-w} &\le \inprod{F(u_{k+1})}{u_{k+1}-w} \nonumber\\
	&\le \frac{\lambda_k}{2}\bracket{ \gnorm{w-w_k}{}{2} - \gnorm{w-u_{k+1}}{}{2} } \nonumber\\
	&\quad+ \underbrace{\nu + \frac{3L^2D^2}{\lambda_k} + \frac{3}{2\lambda_k}\gnorm{F(w_k) - F_{\bB_{k+1}}(w_{k})}{}{2} + 
	\inprod{F(w_k) - F_{\bB_{k+1}}(w_{k})}{w_k-w}}_{E_k}
	\label{eq:int_rel_prox_3},
\end{align}
Letting $u_0 := w_0$ and consequently $\boldsymbol{\xi}_0 = {\bf 0}$, we have from \eqref{eq:int_rel_prox_3}
\begin{align}
	\inprod{F(w)}{u_{k+1}-w} &\le 	\inprod{F(u_{k+1})}{u_{k+1}-w}\nonumber\\  &\le \frac{\lambda_k}{2} \bracket{\gnorm{w-u_k}{}{2} - \gnorm{w-u_{k+1}}{}{2} + 2\inprod{w-u_k}{u_k-w_k} + \gnorm{u_k-w_k}{}{2}} + E_k \nonumber\\
	&=\frac{\lambda_k}{2} \bracket{\gnorm{w-u_k}{}{2} - \gnorm{w-u_{k+1}}{}{2} } + \lambda_k \inprod{\boldsymbol{\xi}_k}{u_k-w} + \frac{1}{2}\lambda_k\gnorm{\boldsymbol{\xi}_k}{}{2} + E_k \label{eq:int_rel_prox_5}
\end{align}
Let us define an auxiliary sequence $\{z_k\}_{k \ge 0}$, where $z_0 = w_0$ and for all $ k \ge 1$, we have
\[z_k = \proj_{\W}[z_{k-1} - \boldsymbol{\xi}_k].\]
Then, due to the mirror-descent bound, we have
\begin{equation}
	\tsum_{k =1}^K \inprod{\boldsymbol{\xi}_k}{z_k - w} \le \frac{1}{2}\gnorm{w-w_0}{}{2} - \frac{1}{2}\gnorm{w-z_K}{}{2} + \tsum_{k =1}^K \gnorm{\boldsymbol{\xi}_k}{}{2}.\label{eq:int_rel_prox_4}
\end{equation}
Moreover, noting that 
\begin{align*}
	\inprod{\boldsymbol{\xi}_k}{u_k - z_k}  &= \inprod{\boldsymbol{\xi}_k}{u_k - z_{k-1}}+\inprod{\boldsymbol{\xi}_k}{ z_{k-1} -z_k} \le \inprod{\boldsymbol{\xi}_k}{u_k - z_{k-1}} + \gnorm{\boldsymbol{\xi}_k}{}{}\gnorm{z_k - z_{k-1}}{}{}\\
	&\le \inprod{\boldsymbol{\xi}_k}{u_k - z_{k-1}} + \gnorm{\boldsymbol{\xi}_k}{}{2}.
\end{align*}
Combining above relation with \eqref{eq:int_rel_prox_4}, we have
\begin{equation}\label{eq:int_rel_prox_19}
	\tsum_{k =1}^K\inprod{\boldsymbol{\xi}_k}{u_k - w} \le \frac{1}{2}\gnorm{w-w_0}{}{2} + 2\tsum_{k =1}^K\gnorm{\boldsymbol{\xi}_k}{}{2} + \tsum_{k =1}^K\inprod{\boldsymbol{\xi}_k}{u_k - z_{k-1}}.
\end{equation}
Now multiplying \eqref{eq:int_rel_prox_7} by $\gamma_k$ then summing from $k = 0$ to $K$; noting the definition of $\wb{w}_K$ and $\Gamma_K$; and using \eqref{eq:int_rel_prox_6}, \eqref{eq:int_rel_prox_5} and \eqref{eq:int_rel_prox_19} along with assumption \eqref{eq:step_size_condition} implies
\begin{align}
	\Gamma_k\inprod{F(w)}{\wb{w}_K - w} &\le \gamma_0\lambda_0 \gnorm{w-w_0}{}{2} \nonumber\\
	&\quad +\tsum_{k = 0}^K \bracket{\frac{\gamma_kM^2}{2\lambda_k} + \frac{\gamma_0\lambda_0}{2} \gnorm{\boldsymbol{\xi}_{k+1}}{}{2}  + \frac{5\gamma_0\lambda_0}{2} \gnorm{\boldsymbol{\xi}_k}{}{2} +\gamma_kE_k + \gamma_0\lambda_0\inprod{\boldsymbol{\xi}_k}{u_k - z_{k-1}}}\nonumber\\
	\Rightarrow 	\Gamma_k\inprod{F(w)}{\proj_{\W}[\wb{w}_K] - w} &\le \gamma_0\lambda_0 \gnorm{w-w_0}{}{2} +\Gamma_K \inprod{F(w)}{\proj_{\W}[\wb{w}_K] - \wb{w}_K} \nonumber\\ 
	&\quad +\tsum_{k = 0}^K \bracket{\frac{\gamma_kM^2}{2\lambda_k} + \frac{\gamma_0\lambda_0}{2} \gnorm{\boldsymbol{\xi}_{k+1}}{}{2}  + \frac{5\gamma_0\lambda_0}{2} \gnorm{\boldsymbol{\xi}_k}{}{2} +\gamma_kE_k + \gamma_0\lambda_0\inprod{\boldsymbol{\xi}_k}{u_k - z_{k-1}}}\nonumber\\
	&\le \gamma_0\lambda_0 \gnorm{w-w_0}{}{2} + \Gamma_KM\dist_\W(\wb{w}_K) \nonumber\\ 
	&\quad +\tsum_{k = 0}^K \bracket{\frac{\gamma_kM^2}{2\lambda_k} + \frac{\gamma_0\lambda_0}{2} \gnorm{\boldsymbol{\xi}_{k+1}}{}{2}  + \frac{5\gamma_0\lambda_0}{2} \gnorm{\boldsymbol{\xi}_k}{}{2} +\gamma_kE_k + \gamma_0\lambda_0\inprod{\boldsymbol{\xi}_k}{u_k - z_{k-1}}}\label{eq:int_rel_prox_12}
\end{align}
Now note that 
\begin{align}
	\tsum_{k = 0}^K \gamma_kE_k &= \Gamma_K \nu + \frac{3L^2D^2}{\gamma_0\lambda_0}\tsum_{k = 0}^K \gamma_k^2 + \frac{3}{2\gamma_0\lambda_0}\tsum_{k = 0}^K \gamma_k^2\gnorm{F_{\bB_{k+1}}(w_{k})-F(w_k)}{}{2}\nonumber\\
	&\quad+ \gamma_0\lambda_0 \tsum_{k = 0}^K \inprod{\frac{1}{\lambda_k}(F(w_k) - F_{\bB_{k+1}}(w_{k}))}{w_k-w}\label{eq:int_rel_prox_9}
\end{align}
Define $\boldsymbol{\Delta}_k := \frac{1}{\lambda_k}(F(w_k) - F_{\bB_{k+1}}(w_{k}))$. Note that $\EE_{\bB_{k+1}}[\boldsymbol{\Delta}_k| w_k] = 0$. Moreover, define an auxiliary sequence $\{h_k\}_{k \ge 0}$ with $h_0 := w_0$ and 
\[h_{k+1} := \proj_{\W}[h_k - \boldsymbol{\Delta}_k].\] Then due to mirror descent bound, we have
\begin{align}
	\tsum_{k = 0}^K \inprod{\boldsymbol{\Delta}_k}{h_{k+1}-w} \le \frac{1}{2}\gnorm{w-w_0}{}{2} - \frac{1}{2}\gnorm{w-h_{K+1}}{}{2} + \tsum_{k = 0}^K\gnorm{\boldsymbol{\Delta}_k}{}{2}.\label{eq:int_rel_prox_8}
\end{align}
Moreover, 
\begin{align*}
	\inprod{\boldsymbol{\Delta}_k}{h_k-h_{k+1}} \le \gnorm{\boldsymbol{\Delta}_k}{}{}\gnorm{h_k-h_{k+1}}{}{} \le \gnorm{\boldsymbol{\Delta}_k}{}{2}.
\end{align*}
Using above relation along with \eqref{eq:int_rel_prox_8}, we have 
\begin{align*}
	\tsum_{k = 0}^K\inprod{\boldsymbol{\Delta}_k}{h_k - w} &= \tsum_{k = 0}^K[ \inprod{\boldsymbol{\Delta}_k}{h_{k+1} - w} + \inprod{\boldsymbol{\Delta}_k}{h_k - h_{k+1}}]\\
	&\le \frac{1}{2}\gnorm{w-w_0}{}{2} - \frac{1}{2}\gnorm{w-h_{K+1}}{}{2} + 2\tsum_{k = 0}^K\gnorm{\boldsymbol{\Delta}_k}{}{2}\\
	\Rightarrow\tsum_{k = 0}^K\inprod{\boldsymbol{\Delta}_k}{w_k - w}&= \tsum_{k = 0}^K[ \inprod{\boldsymbol{\Delta}_k}{w_k-h_k} + \inprod{\boldsymbol{\Delta}_k}{h_k-w}]\\
	&\le \tsum_{k=0}^K\inprod{\boldsymbol{\Delta}_k}{w_k-h_k} + \frac{1}{2}\gnorm{w-w_0}{}{2} - \frac{1}{2}\gnorm{w-h_{K+1}}{}{2} + 2\tsum_{k = 0}^K\gnorm{\boldsymbol{\Delta}_k}{}{2}.
\end{align*}
Using the above relation in \eqref{eq:int_rel_prox_9}, we have
\begin{align}
	\tsum_{k = 0}^K \gamma_kE_k &= \Gamma_K \nu + \frac{3L^2D^2}{\gamma_0\lambda_0}\tsum_{k = 0}^K \gamma_k^2 + \frac{3}{2\gamma_0\lambda_0}\tsum_{k = 0}^K \gamma_k^2\gnorm{F_{\bB_{k+1}}(w_{k})-F(w_k)}{}{2}\nonumber\\
	&\quad+\gamma_0\lambda_0 \tsum_{k =0}^K \inprod{\boldsymbol{\Delta}_k}{w_k-w} \nonumber\\
	&\le \frac{\gamma_0\lambda_0}{2}\gnorm{w-w_0}{}{2} + \Gamma_K\nu \frac{3L^2D^2}{\gamma_0\lambda_0}\tsum_{k = 0}^K \gamma_k^2 + \frac{3}{2\gamma_0\lambda_0}\tsum_{k = 0}^K \gamma_k^2\gnorm{F_{\bB_{k+1}}(w_{k})-F(w_k)}{}{2}\nonumber\\ 
	&\quad+ 2\gamma_0\lambda_0\tsum_{k = 0}^{K} \gnorm{\boldsymbol{\Delta}_k}{}{2} + \gamma_0\lambda_0\tsum_{k = 0}^K \inprod{\boldsymbol{\Delta}_k}{w_k-h_k}.\label{eq:int_rel_prox_13}
\end{align}
Finally, note that for all valid $k$, we have
\begin{align}
	\EE[\gnorm{\bg_k}{}{2}] &= \sigma_k^2d\label{eq:int_rel_prox_10},\\
	\EE[\gnorm{\boldsymbol{\Delta}_k}{}{2}] &=\EE_{w_k}[\EE_{\bB_{k+1}}[\gnorm{\boldsymbol{\Delta}_k}{}{2}| w_k]] \nonumber\\
	&= \frac{1}{\lambda_k^2}\EE_{w_k} \EE_{\bB_{k+1}}\bracket{\gnorm{F_{\bB_{k+1}}(w_{k}) - F(w_k)}{}{2} | w_k}\nonumber\\
	&\le \frac{1}{\lambda_k^2}\EE_{w_k}\EE_{\bB_{k+1}}\gnorm{F_{\bB_{k+1}}(w_{k})}{}{2} \le \frac{M^2}{\lambda_k^2}\label{eq:int_rel_prox_11},\\
	\EE[\inprod{\boldsymbol{\Delta}_k}{w_k-h_k}] &= \EE[\inprod{\EE[\boldsymbol{\Delta}_k|w_k, h_k]}{w_k-h_k}] = 0 \label{eq:int_rel_prox_14},\\
	\EE[\inprod{\boldsymbol{\xi}_k}{u_k-z_{k-1}}] &= \EE[\inprod{\EE[\boldsymbol{\xi}_k| u_k, z_{k-1}]}{u_k-z_{k-1}}] = 0\label{eq:int_rel_prox_15},
\end{align}
where, in \eqref{eq:int_rel_prox_11}, we used the fact that $F_{\bbeta})$ is $M$-bounded for all $\boldsymbol{\beta} \in \bS$. 

Now, using \eqref{eq:int_rel_prox_13} in relation \eqref{eq:int_rel_prox_12}, noting the bound on $\dist_\W(\wb{w}_K)$ from Proposition \ref{prop:bound_for_projected_iterate} (in particular \eqref{eq:int_rel_prox_24}), taking supremum with respect to $w \in \W$, then taking expectation and noting \eqref{eq:int_rel_prox_10}-\eqref{eq:int_rel_prox_15}, we have
\begin{align*}
	\Gamma_k\EE\sup_{w \in \W}\inprod{F(w)}{\proj_{\W}[\wb{w}_K] - w} &\le \frac{3\gamma_0\lambda_0}{2} D^2 + \tsum_{k =0}^K\bracket{\frac{\gamma_k(4M^2+3L^2D^2)}{\lambda_k} + \frac{5\gamma_0\lambda_0}{2}\sigma_{k+1}^2d+ \gamma_k\nu} \\
	&\quad+ \Gamma_K M \sqrt{\frac{1}{\Gamma_K^2}\tsum_{k = 0}^K\gamma_k^2\sigma_{k+1}^2d}.\\
	\Rightarrow \EE\sup_{w \in \W}\inprod{F(w)}{\proj_{\W}[\wb{w}_K] - w} &\le \nu +  \frac{1}{\Gamma_K} [\frac{3\gamma_0\lambda_0}{2}D^2 +  \frac{4M^2 + 3L^2D^2}{\gamma_0\lambda_0}\tsum_{k = 0}^K\gamma_k^2 + \frac{5\gamma_0\lambda_0d}{2}\tsum_{k =1}^K\sigma_{k+1}^2 ]\\
	&\quad + M \sqrt{\frac{1}{\Gamma_K^2}\tsum_{k = 0}^K\gamma_k^2\sigma_{k+1}^2d},
\end{align*}
where in the first inequality, we used the fact that $\EE[\dist_\W(\wb{w}_K)] \le \sqrt{\EE[\dist_\W(\wb{w}_K)^2]}$. Hence, we conclude the proof of \eqref{eq:conv_rate_svi_risk}. 

Now, we extend this for \eqref{spp_problem}. Denote $u^{k+1} = (\wt{x}^{k+1}, \wt{y}^{k+1})$. Then, we have
\begin{align*}
	\inprod{F(u_{k+1})}{u_{k+1} -w}&= \inprod{\grad_x f(\wt{x}_{k+1}, \wt{y}_{k+1} )}{\wt{x}_{k+1} - x} + \inprod{-\grad_y f(\wt{x}_{k+1}, \wt{y}_{k+1} )}{\wt{y}_{k+1} - y}\\
	&\ge f(\wt{x}_{k+1}, \wt{y}_{k+1}) - f(x, \wt{y}_{k+1})  + [-f(\wt{x}_{k+1}, \wt{y}_{k+1}) + f(\wt{x}_{k+1}, y)]\\
	&= f(\wt{x}_{k+1}, y) - f(x, \wt{y}_{k+1}). 
\end{align*}
	Using the above in \eqref{eq:int_rel_prox_5}, we obtain,
\begin{align*}
	f(\wt{x}_{k+1}, y) - f(x, \wt{y}_{k+1}) &\le \frac{\lambda_k}{2} \bracket{\gnorm{w-u_k}{}{2} - \gnorm{w-u_{k+1}}{}{2} } + \lambda_k \inprod{\boldsymbol{\xi}_k}{u_k-w} + \frac{1}{2}\lambda_k\gnorm{\boldsymbol{\xi}_k}{}{2} + E_k. 
\end{align*}
Now, using Proposition \ref{prop:bound_for_projected_iterate} to bound the distance between points $\tfrac{1}{\Gamma_K}(\tsum_{k =0}^K\gamma_k\wt{x}_{k+1}, \tsum_{k =0}^K\gamma_k\wt{y}_{k+1})$ and $(\Pi_\X[\wb{x}_K], \Pi_\Y[\wb{y}_K])$ and using Jensen's inequality to conclude that 
\begin{align*}
	\frac{1}{\Gamma_K}\tsum_{k = 0}^K\gamma_k [f(\wt{x}_{k+1}, y) - f(x, \wt{y}_{k+1})] &\ge f(\tfrac{1}{\Gamma_K}\tsum_{k = 0}^K\gamma_k \wt{x}_{k+1}, y) - f(x, \tfrac{1}{\Gamma_K} \tsum_{k = 0}^K\gamma_k\wt{y}_{k+1})]\\
	&\ge f(\wb{x}_K,y) - f(x, \wb{y}_K) - M\gnorm*{\begin{bmatrix}
			\tfrac{1}{\Gamma_K}\tsum_{k = 0}^K\gamma_k \wt{x}_{k+1} - \proj_\X[\wb{x}_K]\\
			\tfrac{1}{\Gamma_K}\tsum_{k = 0}^K\gamma_k \wt{y}_{k+1} - \proj_\Y[\wb{y}_K]
	\end{bmatrix}}{}{}
\end{align*} and retracing the steps of this proof from \eqref{eq:int_rel_prox_5}, we obtain \eqref{eq:conv_spp_risk}. Hence, we conclude the proof.
\endgroup
\end{proof}

\subsection{Differential privacy of the NISPP method}
First, we show a simple bound on $\ell_2$-sensitivity for updates of NISPP method.
\begin{proposition}\label{prop:l_2-sensitivity_prox}
	Suppose $\nu \le \frac{2M^2}{\lambda_kB_{k+1}^2}$ then $\ell_2$-sensitivity of updates of Algorithm \ref{alg:alg0} 
	is at most $\frac{4M}{\lambda_kB_{k+1}}$
	where $B_{k+1} = \abs{\bB_{k+1}}$ is the batch size of $k$-th iteration.
\end{proposition}
\begin{proof}
	Let $w_k$ be an iterate in the start of $k$-th iteration of Algorithm \ref{alg:alg0}. Suppose $\bB_{k+1}$ and $\bB_{k+1}'$ be two different batches used in $k$-th iteration to obtain $u_{k+1}$ and $u_{k+1}'$, respectively. Also note that $\bB_{k+1}$ and $\bB_{k+1}'$ differ in only single datapoint. Then, due to \eqref{eq:int_rel_prox_1}, we have for all $w \in \W$
	\begin{align*}
		\inprod{F_{\bB_{k+1}}(u_{k+1}) + \lambda_k(u_{k+1} - w_k) }{w - u_{k+1}} &\ge -\nu\\
		\inprod{F_{\bB_{k+1}'}(u_{k+1}') + \lambda_k(u_{k+1}' - w_k) }{w - u_{k+1}'} &\ge -\nu
	\end{align*}
	Using $w = u_{k+1}'$ in the first relation and $w = u_{k+1}$ in the second relation above and then summing, we obtain
	\begin{align*}
		\inprod{ F_{\bB_{k+1}}(u_{k+1}) -F_{\bB_{k+1}'}(u_{k+1}') }{ u_{k+1} - u_{k+1}'} &\le 2\nu - \lambda_k \gnorm{u_{k+1} - u_{k+1}'}{}{2}\\
		\Rightarrow \inprod{F_{\bB_{k+1}}(u_{k+1}) - F_{\bB_{k+1}}(u_{k+1}') }{ u_{k+1} - u_{k+1}'}&\le \inprod{  F_{\bB_{k+1}'}(u_{k+1}')- F_{\bB_{k+1}}(u_{k+1}') }{ u_{k+1} - u_{k+1}'} \\
		&\quad+ 2\nu - \lambda_k \gnorm{u_{k+1} - u_{k+1}'}{}{2}
	\end{align*}
	Now, noting that $F_{\bB_{k+1}}$ is a monotone operator and denoting $\mathbf{a}_{k+1} := \gnorm{ F_{\bB_{k+1}'}(u_{k+1}')- F_{\bB_{k+1}}(u_{k+1}')}{}{}$, $p_{k+1} := \gnorm{w_{k+1} - w'_{k+1}}{}{} = \gnorm{u_{k+1} - u'_{k+1}}{}{}$ we have
	\begin{align}
		0 &\le \inprod{  F_{\bB_{k+1}'}(u_{k+1}')- F_{\bB_{k+1}}(u_{k+1}') }{ u_{k+1} - u_{k+1}'}+ 2\nu - \lambda_k \gnorm{u_{k+1} - u_{k+1}'}{}{2}\nonumber\\
		&\le \mathbf{a}_{k+1} p_{k+1} - \lambda_k p_{k+1}^2 + 2\nu. \label{eq:int_rel_prox_16}
	\end{align}
	Finally noting that if $\boldsymbol{\beta}$ and $\boldsymbol{\beta}'$ are the differing datapoints in $\bB_{k+1}$ and $\bB_{k+1}'$, then 
	\begin{align*}
		\mathbf{a}_{k+1} &= \frac{1}{B_{k+1}}\gnorm{F_{\bbeta'}(u'_{k+1}) - F_{\bbeta}(u'_{k+1}) }{}{} \le \frac{2M}{B_{k+1}}.
	\end{align*}
	Using the above relation in \eqref{eq:int_rel_prox_16} and noting that $\ell_2$-sensitivity $p_{k+1} = \gnorm{w_{k+1}-w_{k+1}'}{}{} = \gnorm{u_{k+1} -u_{k+1}'}{}{}$, we have, $p_{k+1}$ satisfies
	\[p_{k+1}^2 -\frac{2M}{\lambda_k B_{k+1}} p_{k+1} - \frac{2\nu}{\lambda_k} \le 0.\]
	This implies
	\[p_{k+1} \le \frac{M}{\lambda_kB_{k+1}} + \sqrt{\frac{M^2}{\lambda_k^2B_{k+1}^2} + \frac{2\nu}{\lambda_k}} \le \frac{2M}{\lambda_kB_{k+1}} + \sqrt{\frac{2\nu}{\lambda_k}}.\]
	Setting $\nu \le \frac{2M^2}{\lambda_kB_{k+1}^2}$, we have $p_{k+1} \le \frac{4M}{\lambda_kB_{k+1}}.$
	Hence, we conclude the proof.
\end{proof}
Using the $\ell_2$-sensitivity result above along with Proposition \ref{propos:Gauss_mech} and \ref{propos:parallel_comp}, we immediately obtain the following:
\begin{proposition}\label{thm:dp_proof_prox_onepass}
	Algorithm \ref{alg:alg0} with batch sizes $(B_{k+1})_{k \in [K]_0}$, parameters $(\lambda_k)_{k \in [K]_0}$, variance $\sigma_{k+1}^2 = \frac{32M^2}{\lambda_k^2B_{k+1}^2}\frac{\ln(1/\eta)}{\eps^2}$ and $\nu$ satisfying assumptions of Proposition \ref{prop:l_2-sensitivity_prox} is $(\eps, \eta)$-differentially private. 
\end{proposition}

Now, we provide a policy for setting $\gamma_k,  \lambda_k$ and $B_{k+1}$ to obtain population risk bounds for DP-SVI and DP-SSP problem by the NISPP method.
\begin{corollary}\label{cor:stepsize_policy_prox_svi}
	Algorithm \ref{alg:alg0} with disjoint batches $\bB_{k+1}$ is of size $B_{k+1} = B := n^{1/3}$ for all $k \ge 0$ and the following parameters
	\begin{alignat*}{2}
		\gamma_k &= 1, \hspace{8em}
		&&\lambda_k = \lambda_0 := \max\{\frac{M}{D}, L\}\max\braces{n^{1/3}, \frac{\sqrt{d\ln(1/\eta)}}{\eps}},\\
		\sigma_{k+1}^2 &= \frac{32M^2}{ B\lambda_0} \frac{\ln(1/\eta)}{\eps^2},
		&&\nu = \frac{2M^2}{\lambda_0B^2},
	\end{alignat*}
	 is  $(\eps, \eta)$-differentially private and achieves expected SVI-gap (SSP-gap, respectively)
	\[
	 O\paran*{(M+LD)D\bracket*{\frac{1}{n^{1/3}} + \frac{\sqrt{d\ln(1/\eta)}}{\eps n^{2/3}}} 
	 }.
	 \]
\end{corollary}
\begin{proof}
	Note that values of $\nu$, $\sigma_{k+1}$ and other required conditions proposed in Propositions \ref{prop:l_2-sensitivity_prox} and \ref{thm:dp_proof_prox_onepass} are satisfied. Hence, this algorithm is $(\eps, \eta)$-differentially private.
	
	Moreover, all requirements of Theorem \ref{thm:conv_prox_vi} are satisfied. In order to maintain single pass over the dataset, we require $K = \frac{n}{B} = n^{2/3}$ iterations. Then,  we provide individual bounds on the terms of \eqref{eq:conv_rate_svi_risk} (\eqref{eq:conv_spp_risk}, respectively) and conclude the corollary using Theorem \ref{thm:conv_prox_vi}.
	
	Note that we are using a constant parameter policy. Hence, $\sigma_{k+1} = \sigma = \frac{4M}{\rho B\lambda_0}$ for all $k \ge 0$. Substituting appropriate parameter values, we have 
	\begin{align*}
		\nu &= \frac{2MD}{n^{2/3}\max\{n^{1/3}, \sqrt{d\ln(1/\eta)}/\eps\} } \le \frac{2MD}{n},\\
		M \sqrt{\frac{1}{\Gamma_K^2}\tsum_{k = 0}^K\gamma_k^2\sigma_{k+1}^2d}&= \frac{M\sqrt{d}\sigma_{k+1}}{\sqrt{K}} = \frac{4M^2\sqrt{2d\ln(1/\eta)}}{\eps n^{2/3}\lambda_0} \le \frac{4\sqrt{2}MD}{n^{2/3}}, \\
		\frac{3\lambda_0D^2}{2K} &\le \frac{3(M+LD)D}{2} \paran*{\frac{1}{n^{1/3}} + \frac{\sqrt{d\ln(1/\eta)}}{\eps n^{2/3}}},\\
		\frac{4M^2 + 3L^2D^2}{\lambda_0} &\le 
		\frac{4MD}{n^{1/3}} + \frac{3LD^2}{n^{1/3}},\\
		\frac{5\lambda_0d\sigma^2}{2} &= \frac{40M^2d\ln(1/\eta)}{\eps^2B^2 \lambda_0} \le \frac{40MD \sqrt{d\ln(1/\eta)}}{\eps n^{2/3}}.
	\end{align*}
	Substituting these bounds in Theorem \ref{thm:conv_prox_vi}, we conclude the proof.
\end{proof}

\begin{remark}\label{rem:prox_runtime}We have the following remarks for NISPP method:
	\begin{enumerate}
		\item In order to obtain $\nu$-approximate solution of the subproblem of NISPP method satisfying \eqref{eq:int_rel_prox_1}, we can use the Operator Extrapolation (OE) method (see Theorem 2.3 \cite{Kotsalis:2020}). OE method outputs a solution $u_{k+1}$ satisfying $\gnorm{u_{k+1} -w^*_{k+1}}{}{} \le \zeta$ in $ \frac{L+\lambda_0}{\lambda_0}\ln(\frac{D}{\zeta})$ iterations, where $w_{k+1}^*$ is an exact SVI solution for problem \eqref{eq:int_rel_prox_1}. Furthermore, we have for all $w \in \W$, 
		\begin{align*}
			0 &\le \inprod{F(w^*_{k+1}) + \lambda_k(w^*_{k+1}-w_k)}{w-w^*_{k+1}} \\
			 &= \inprod{F(u_{k+1}) + F(w^*_{k+1}) - F(u_{k+1}) + \lambda_k (u_{k+1}- w_k) + \lambda_k(w^*_{k+1}- u_{k+1}) }{w -w^*_{k+1}} \\
			&\le \inprod{F(u_{k+1}) + \lambda_k (u_{k+1}- w_k)}{w -w^*_{k+1}} + (L+\lambda_k)\gnorm{u_{k+1}-w^*_{k+1}}{}{}\gnorm{w -w^*_{k+1}}{}{}\\
			&\le  \inprod{F(u_{k+1}) + \lambda_k (u_{k+1}- w_k)}{w -w^*_{k+1}} + (L+\lambda_k)D\gnorm{u_{k+1}-w^*_{k+1}}{}{}\\
			&=  \inprod{F(u_{k+1}) + \lambda_k (u_{k+1}- w_k)}{w- u_{k+1} + u_{k+1}-w^*_{k+1}} + (L+\lambda_k)D\gnorm{u_{k+1}-w^*_{k+1}}{}{}\\
			&\le \inprod{F(u_{k+1}) + \lambda_k (u_{k+1}- w_k)}{w- u_{k+1}} + \gnorm{F(u_{k+1}) + \lambda_k (u_{k+1}- w_k)}{}{}\gnorm{u_{k+1}-w^*_{k+1}}{}{}\\
			&\quad + (L+\lambda_k)D\gnorm{u_{k+1}-w^*_{k+1}}{}{}\\
				&\le \inprod{F(u_{k+1}) + \lambda_k (u_{k+1}- w_k)}{w- u_{k+1}} + [LD+M+2\lambda_kD]\gnorm{u_{k+1}-w^*_{k+1}}{}{}
		\end{align*}
		Setting 
		$\zeta=\nu/[LD+M+2\lambda_kD]$, we obtain that $u_{k+1}$ is a $\nu$-approximate solution satisfying \eqref{eq:int_rel_prox_1}. Using the convergence rate above, we require $\frac{L+\lambda_0}{\lambda_0}\ln\frac{MD+LD^2+2\lambda_kD^2}{\nu}$ operator evaluations.
		
		Note that since, $\lambda_0 \ge L$, 
		we have $\frac{L+\lambda_0}{\lambda_0} \le 2$. 
		Moreover, 
		\begin{align}
			\ln{\frac{MD+LD^2 + 2\lambda_kD^2}{\nu}} &\le \ln{\frac{4\lambda_kD^2}{\nu}}\label{eq:int_rel_prox_29}\\
			&= \ln\paran[\big]{\frac{2\lambda_0^2D^2B^2}{M^2}} \nonumber\\
			&= \ln\paran[\big]{ n^{2/3}\max\braces[\big]{n^{2/3}, \frac{d\ln(1\privb)}{\priva^2}}\max\braces[\big]{1, \frac{L^2D^2}{M^2}}  }\nonumber
		\end{align} 
	Hence, each iteration of NISPP method requires $O(\log{n})$ iterations of OE method for solving the subproblem. Moreover, each iteration of the OE method requires $2B$ stochastic operator evaluations. Hence, we require $O(KB\log{n})$ stochastic operator evaluations in the entire run of NISPP (Algorithm \ref{alg:alg0}). Noting that $KB = n$, we conclude that this is a near linear time algorithm and also performs only a single pass over the data in the stochastic outer-loop. We provide the details of OE method in the Appendix \ref{sec:OE_method}.
		\item For non-DP version of NISPP method, i.e., $\sigma_k = 0$ for all $k$, we can easily obtain population risk bound of $O(\frac{MD}{\sqrt{n}})$ by setting $\lambda_0 = \frac{M}{D}\sqrt{n}$, $B = 1$ (or $K = n$) and $\nu = \frac{MD}{\sqrt{n}}$ in Corollary \ref{cor:stepsize_policy_prox_svi}.
	\end{enumerate}
\end{remark}

	In view of Corollary \ref{cor:stepsize_policy_prox_svi}, it seems that running NISPP method for $n^{3/2}$ stochastic operator evaluations may provide optimal risk bounds. However, running that many stochastic operator evaluations requires multi-pass over the dataset so, in principle, this would only provide bounds in the empirical risk. 
In order to compute the population risk of this multi-pass version, we analyze the stability of NISPP and provide generalization guarantees which result in optimal population risk.

%% file: sec_prox_conv_SPP.tex
\section{Stability of NISPP and Optimal Risk for DP-SVI  and DP-SSP} \label{sec:stab_risk_NISPP}
In this section, we develop a multi-pass variant of NISPP method, and prove its stability to extrapolate empirical performance to population risk bounds.

%% file: sec_stab_prox.tex
\subsection{Stability of NISPP method} \label{sec:stab_prox}

Let us start with two adjacent datasets $\bS \simeq \bS'$. Suppose we run NISPP method on both datasets starting from the same point $w_0 \in \W$. 
Then, in the following lemma, we provide bound on the how far apart trajectories of these two runs can drift. 
\begin{lemma} \label{lem:rec_stab_PPM}
	Let $(u_{k+1}, w_{k+1})_{k\ge 0}$ and $(u'_{k+1}, w'_{k+1})_{k\ge 0}$ be two trajectories of the NISPP method (Algorithm \ref{alg:alg0}) for any adjacent datasets $\bS \simeq \bS'$ whose batches are denotes by $\bB_{k+1}, \bB'_{k+1}$ respectively. Moreover, denote $\mathbf{a}_{k+1}  :=  \gnorm{F_{\bB_{k+1}}(u_{k+1}) - F_{\bB'_{k+1}}(u_{k+1})}{}{}$ and $\delta_{k+1} := \gnorm{u_{k+1}-u'_{k+1}}{}{} (= \gnorm{w_{k+1}-w'_{k+1}}{}{})$ for  $k$-th iteration of Algorithm \ref{alg:alg0}. Then, if $i = \inf\{k : \bB_{k+1} \ne \bB'_{k+1}\}$, 
	\begin{equation}
		\delta_{j+1} \begin{cases}
			=0 &\text{if } j+1 \le i\\
			\le \sum_{k = i}^j\frac{2\mathbf{a}_{k+1}}{\lambda_k} + \sqrt{\frac{4\nu}{\lambda_k}} &\text{otherwise}.
		\end{cases}\label{eq:int_rel_prox_28}
	\end{equation}
\end{lemma}
\begin{proof}
	It is clear from the definition $i$ that $\bB_j = \bB_j'$ for all $j \le i$. This implies $u_j = u'_j$ and $w_j = w_j'$ for all $j \le i$. Hence, we conclude first case of \eqref{eq:int_rel_prox_28}.
	
	Using \eqref{eq:int_rel_prox_1} for $\nu$-approximate strong VI solution, we have, 
	\begin{align}
		\inprod{F_{\bB_{k+1}}(u_{k+1}) + \lambda_k(u_{k+1}-w_k)}{w - u_{k+1}} &\ge -\nu\label{eq:int_rel_prox_26},\\
		\inprod{F_{\bB'_{k+1}}(u'_{k+1}) + \lambda_k(u'_{k+1}-w'_k)}{w - u'_{k+1}} &\ge -\nu\label{eq:int_rel_prox_27}.
	\end{align}Then, adding \eqref{eq:int_rel_prox_26} with $w = u'_{k+1}$ and \eqref{eq:int_rel_prox_27} with $w = u_{k+1}$, we have 
	\begin{align}
		\inprod{F_{\bbeta_{k+1}}(u_{k+1}) - F_{\bbeta'_{k+1}}(u'_{k+1})}{u_{k+1}-u'_{k+1}} &\le 2\nu - \lambda_k\inprod{u_{k+1} -u'_{k+1}}{(u_{k+1}-w_k) - (u'_{k+1}-w'_k)}\nonumber\\
		&=2\nu - \lambda_k \delta_{k+1}^2  + \lambda_k\inprod{u_{k+1} -u'_{k+1}}{w_k - w_k'} \nonumber\\
		&\le 2\nu - \lambda_k\delta_{k+1}^2 + \frac{\lambda_k}{2}[\delta_{k}^2 + \delta_{k+1}^2]\nonumber\\
		&\le 2\nu  -\frac{\lambda_k}{2}\delta_{k+1}^2 + \frac{\lambda_k}{2}\delta_{k}^2 \label{eq:int_rel_prox_25}. 
	\end{align}
Also note that 
\begin{align*}
	\inprod{F_{\bB_{k+1}}(u_{k+1}) - F_{\bB'_{k+1}}(u'_{k+1})}{u_{k+1}-u'_{k+1}} &= \inprod{F_{\bB'_{k+1}}(u_{k+1}) - F_{\bB'_{k+1}}(u'_{k+1})}{u_{k+1}-u'_{k+1}} \\
	&\quad+ \inprod{F_{\bB_{k+1}}(u_{k+1}) - F_{\bB'_{k+1}}(u_{k+1})}{u_{k+1}-u'_{k+1}}\\
	&\ge \inprod{F_{\bB_{k+1}}(u_{k+1}) - F_{\bB'_{k+1}}(u_{k+1})}{u_{k+1}-u'_{k+1}}
\end{align*}
where the last inequality follows from monotonicity of $F_{\bB'_{k+1}}$. Using above relation along with \eqref{eq:int_rel_prox_25}, we obtain 
\begin{align*}
	\frac{\lambda_k}{2} \delta_{k+1}^2 &\le \frac{\lambda_k}{2}\delta_{k}^2 + 2\nu + \inprod{F_{\bB'_{k+1}}(u_{k+1}) - F_{\bB_{k+1}}(u_{k+1})}{u_{k+1}-u'_{k+1}}\\
	\Rightarrow \delta_{k+1}^2 &\le \delta_{k}^2 + \frac{4\nu}{\lambda_k} + \frac{2}{\lambda_k}\mathbf{a}_{k+1}\delta_{k+1},
\end{align*} where we used the definition $a_k$ along with Cauchy-Schwarz inequality. Solving for the quadratic inequality in $\delta_{k+1}$, we obtain the following recursion
\[ \delta_{k+1} \le \frac{\mathbf{a}_{k+1}}{\lambda_k} + \sqrt{\frac{\mathbf{a}_{k+1}^2}{\lambda_k^2} + \delta_{k}^2 + \frac{4\nu}{\lambda_k}}\]
which can be further simplified to 
\[ \delta_{k+1} \le \delta_{k} + \frac{2\mathbf{a}_{k+1}}{\lambda_k} + \sqrt{\frac{4\nu}{\lambda_k}}.\]
Solving this recursion and noting the base case that $\delta_i = 0$, we obtain \eqref{eq:int_rel_prox_28}.
\end{proof}
A direct consequence of the previous analysis are in-expectation and high probability uniform argument stability upper bounds 
for the sampling with replacement variant of Algorithm \ref{alg:alg0}.
\begin{theorem}\label{prop:stability_bound_w}
	Let $\A$ denote the sampling with replacement NISPP method (Algorithm \ref{alg:alg0}) where $\bB_k$ is chosen uniformly at random from subsets of $\bS$ of a given size $B_k$. Then $\A$  
	satisfies the following uniform argument stability bounds: 
	$$ \sup_{\bS\simeq \bS^{\prime}}\EE_{\A}[\delta_{\A}(\bS,\bS^{\prime})] 
	\leq  \sum_{k=1}^{K}\Big( \frac{2M}{n\lambda_k}+\sqrt{\frac{4\nu}{\lambda_k}} \Big) .$$
	
	Furthermore, if $|\bB_{k}|= B$ and $\lambda_k= \lambda$ for all $k$ (i.e., constant batch size and regularization parameter throughout iterations) then w.p.~at least $1-n\exp\{-KB/[4n]\}$ (over both sampling and noise addition) 
	$$ \sup_{\bS\simeq \bS^{\prime}}[\delta_{\A}(\bS,\bS^{\prime})] \leq \frac{4MK}{\lambda n}+K\sqrt{\frac{4\nu}{\lambda}}.$$
\end{theorem}
\begin{proof}
	Let $\bS\simeq \bS^{\prime}$ and $(u_{k+1})_k$, $(u_{k+1}^{\prime})_k$ the trajectories of 
	the algorithm on $\bS$ and $\bS^{\prime}$, respectively. 
	By Lemma \ref{lem:rec_stab_PPM}, letting $\delta_{k+1}=\|\tilde w_{k+1}-\tilde w_{k+1}^{\prime}\|$, we get 
	$\delta_{j} \leq \sum_{k=1}^{j}\Big( \frac{2\mathbf{a}_k}{\lambda_k}+\sqrt{\frac{4\nu}{\lambda_k}} \Big),$
	where $\mathbf{a}_k=\gnorm{F_{\bB_{k+1}}(u_{k+1})-F_{\bB_{k+1}'}(u_{k+1}')}{}{}$ is a random 
	variable. By the law of total probability, $\EE[\mathbf{a}_k]
	\leq \frac{|\bB_{k+1}|}{n}\frac{2M}{|\bB_{k+1}|}+\big(1-\frac{|\bB_{k+1}|}{n}\big)\cdot 0=\frac{2M}{n}.$
	Hence, $\EE[\delta_{j}]\leq \sum_{k=1}^{j}\Big( \frac{2M}{n\lambda_k}+\sqrt{\frac{4\nu}{\lambda_k}} \Big)
	\leq \sum_{k=1}^{K}\Big( \frac{2M}{n\lambda_k}+\sqrt{\frac{4\nu}{\lambda_k}} \Big).$ 
	Since $ \gnorm{\proj_{\W}(\wb{w}_K)-\proj_{\W}(\wb{w}_K^{\prime})}{}{} \le \gnorm{\wb{w}_K - \wb{w}'_K}{}{} \le \max_{k\in[K]}\delta_k$, and since $\bS\simeq \bS^{\prime}$ are arbitrary,
	$$ \sup_{\bS\simeq \bS^{\prime}}\EE_{\A}[\delta_{\cal A}(\bS,\bS^{\prime})] 
	\leq \sum_{k=1}^{K}\Big( \frac{2M}{n\lambda_k}+\sqrt{\frac{4\nu}{\lambda_k}} \Big) .$$

We proceed now to the high-probability bound. Let $\br_k\sim\mbox{Ber}(p)$, for $k\in[K]$,
		with $Kp< 1$. Then, for any $0<\theta<1/2$,
		\begin{eqnarray*}
			\PP\left[  \sum_{k=1}^K\br_k \geq Kp+\tau \right] 
			&\leq& \exp\big(-\theta(\tau+Kp))\Big[ 1+p(e^{\theta}-1) \Big]^K \leq \exp\{Kp\theta^2-\theta\tau\}.
		\end{eqnarray*}
		Choosing $\theta=\tau/[2Kp]<1/2$, we get that the probability above is upper bounded by $\exp\{-\tau^2/[4Kp]\}$.
		Finally, choosing $\tau=Kp$, we get
		$$\PP\Big[  \sum_{k=1}^K\br_k \geq 2Kp \Big] \leq \exp\{-Kp/4\}. $$
	
	Next, fix the coordinate $i$ where $S\simeq S^{\prime}$ may differ. Noticing that $\mathbf{a}_k$ 
	is a.s. upper bounded by $(2M/B)\br_k$ with $\br_k\sim\mbox{Ber}(p)$, with $p=B/n$, we get
	$$ \PP\Big[ \sum_{k=1}^K\frac{2\mathbf{a}_k}{\lambda} \geq \frac{2}{\lambda}\frac{2M}{n}\Big] \leq \exp\{-\frac{KB}{4n}\}. $$
	In particular, w.p. at least~$1-\exp\{-\frac{KB}{4n}\}$, we have 
	$ \mathbf{a}_k \leq \frac{4M}{\lambda n}+\sqrt{\frac{4\nu}{\lambda}}.$ 
	Using a union bound over $i\in[n]$ (and noticing that averaging preserves the stability bound), we conclude that
	$$ \PP\Big[ \sup_{\bS\simeq \bS^{\prime}} \delta_{\A}(\bS,\bS^{\prime}) >\frac{4MK}{\lambda n}+K\sqrt{\frac{4\nu}{\lambda}}\Big] \leq n\exp\{-KB/4n\}. $$
	Hence, we conclude the proof.
\end{proof}
\begin{remark}
	An important observation, for the high-probability guarantee to be useful, is necessary that the algorithm is run for sufficiently many iterations; in particular, we require $K=\omega(n/B)$. Whether this assumption can be completely avoided is an interesting question. Nevertheless, as we will see in the section, our policy for DP-SVI and DP-SSP problem satisfies this requirement.
\end{remark}

\subsection{Optimal risk for DP-SVI and DP-SSP by the NISPP method}

In previous three sections, we provided bounds on optimization error, generalization error and value of $\sigma$ for obtaining $(\priva,\privb)$-differential privacy. In this section, we specify a policy for selecting $\lambda_k, B_k, \gamma_k, \sigma_k $ and $\nu$ such that requirement in previous three sections are satisfied and we can get optimal risk bounds while maintaining $(\priva,\privb)$-privacy. In particular, consider the multi-pass NISPP method where each sample batch $\bB_k$ is chosen randomly from subsets of $\bS$ with replacement. Then, we have the following theorem:
\begin{theorem}\label{thm:multipass_prox_optimal}
	Let $\A$ be the multi-pass NISPP method (Algorithm \ref{alg:alg0}). Set the following constant stepsize and batchsize policy for $\A$:
	\begin{table}[H]
		\centering
		\begin{alignat*}{3}
			&\gamma_k = 1, \qquad  &&\lambda_k = \lambda_0= \max\braces*{\frac{M}{D},L}\max\braces*{\sqrt{n}, \frac{\sqrt{d\log{1/\delta}}}{\theta}}, \qquad &&B_k = B= \sqrt{n},\\
			&K = n,  &&\nu = \frac{M^2}{\lambda_0n^2}, \quad &&\sigma_{k+1} = \frac{8M}{B \lambda_0}\frac{\sqrt{\ln(1/\eta)}}{\eps} .
		\end{alignat*}
	\end{table}\noindent
	Then, Algorithm \ref{alg:alg0} is $(\priva, \privb)$-differential private. Moreover, output $\A(\bS)$ satisfies the following bound on $\EE_{{\cal A}}[\wVIgap{\A(\bS)}{F}]$ for SVI problem (or $\EE_{{\cal A}}[\wSPgap{\A(\bS)}{f}]$ for SSP problem)
	\[ O\paran*{(M+LD)D \paran*{\frac{1}{\sqrt{n}} + \frac{\sqrt{d\ln{1/\eta}}}{\eps n}}},\]
	Moreover, such solution is obtained in total of $\wt{O}(n\sqrt{n})$ stochastic operator evaluations. 
\end{theorem}
\begin{proof}
	
	Note that since $\nu$ satisfies assumption in Proposition \ref{prop:l_2-sensitivity_prox}, we have $\ell_2$-sensitivity of the update of $u_{k+1}$ is $\frac{4M}{\lambda_0B_{k+1}}$. Then, in view of Theorem \ref{thm:dp_prox_multipass} along with value of $\sigma_{k+1}$, we conclude that Algorithm~\ref{alg:alg0} is $(\priva, \privb)$-differential private. 

	Now, for convergence, we first bound the empirical gap. Given that our bounds for \eqref{main_problem} and \eqref{spp_problem} are analogous, we proceed indistinctively for both problems. By Theorem \ref{thm:conv_prox_vi}, along with the fact that sampling with replacement is an unbiased stochastic oracle for the empirical operator, we have for any $\bS$
	\begin{align*}
		\EE_{\A}[\mbox{EmpGap}(\A,F_\bS)] &\le \nu + \frac{\lambda_0D^2}{n} + \frac{7M^2}{\lambda_0} + \frac{160M^2d}{\eps^2B^2\lambda_0}\ln{\frac{1}{\eta}} + M\sqrt{2d}\frac{8M}{Bn \eps \lambda_0}\\
		&\le O\paran*{MD \paran*{\frac{1}{\sqrt{n}} + \frac{\sqrt{d \ln{1/\eta}}}{\eps n} 	} }. 
	\end{align*}
	Similar claims can be made for empirical gap for \eqref{spp_problem} problem: $\EE_{{\cal A}}\big[ \mbox{EmpGap}(\A, f_{\bS})\big]$ where output of $\A$ is $(x(\bS), y(\bS))$.
	
	Next, by Theorem \ref{prop:stability_bound_w}, we have that $\A(\bS)$ (or $x(\bS)$ and $y(\bS)$ for the SSP case) are UAS with parameter
	\[\delta = \frac{2M}{\lambda_0} + n\sqrt{\frac{4\nu}{\lambda_0}} = \frac{4M}{\lambda_0} \le \frac{4D}{\sqrt{n}}. \]
	Hence, noting that empirical risk upper bounds weak VI or SP gap, i.e., using Proposition \ref{prop:weak_generalization_SSP} or Proposition \ref{prop:weak_generalization_SVI} (depending on whether
	the problem is an SSP or SVI, respectively), we have that the risk is upper bounded by its empirical risk plus $M\delta$, where
	$\delta$ is the UAS parameter of the algorithm; in particular, if $\mbox{WeakGap}(\A;\bS)$ is the (SVI or SSP, respectively) gap function for the 
	expectation objective, then
	\begin{align*}
		\EE_{\A, \bS } \big[ \wVIgap{\A}{F}\big] \le O\Big(MD\Big(\frac{1}{\sqrt{n}} + \frac{\sqrt{d\ln(1/\eta)} } {\eps n}\Big)\Big) + \frac{20MD}{\sqrt{n}} = O\Big(MD\Big(\frac{1}{\sqrt{n}} + \frac{\sqrt{d\ln(1/\eta)}}{\eps n}\Big)\Big)
	\end{align*}
	Similar claim can be made for $\wSPgap{\A}{f}$.\\
	Finally, we analyze the running time performance. As in Remark \ref{rem:prox_runtime} , number of OE method iterations for obtaining $\nu$-approximate solution is $O\big(\frac{L+\lambda_0}{\lambda_0}\ln\big(\frac{LD^2+MD+\lambda_0D^2}{\nu}\big)\big)$. Now note that $\frac{L+\lambda_0}{\lambda_0} \le \frac{\sqrt{n}+1}{\sqrt{n}} \le 2$ since $n \ge 1$. Moreover, in view of \eqref{eq:int_rel_prox_29}, we have  $\ln\big(\frac{LD^2+MD+\lambda_0D^2}{\nu}\big) \le \ln\paran[\big]{\frac{4\lambda_0D^2}{\nu}} \le \ln\paran[\big]{n^2\max\braces[\big]{1, \frac{L^2D^2}{M^2}}\max\braces[\big]{n, \frac{d\ln(1/\privb)}{\priva^2}}}$. 
	Each iteration of OE method costs $B$ stochastic operator evaluations and we run outer-loop of NISPP method $K$ times. Hence, total stochastic operator evaluations (after ignoring the $\ln$-term) $\wt{O}(KB) = \wt{O}(n\sqrt{n})$. 
	Hence, we conclude the proof.
\end{proof}


%% file: sec_LB.tex
\section{Lower Bounds and Optimality of our Algorithms} \label{sec:LowerBounds}

In this Section, we show the optimality of our obtained rates from Sections \ref{subsec:stab_NSEG} and \ref{sec:stab_prox}. The first observation is that, since DP-SCO corresponds to a DP-SSP problem where ${\cal Y}$ is a singleton, the complexity of DP-SSP is lower bounded
by $\Omega(MD \big( \frac{1}{\sqrt n}+\min\big\{1,\frac{\sqrt{d}}{\varepsilon n}\big\} \big)\big)$: this is a known lower bound for DP-SCO \cite{Bassily:2019}. \added{It is important to note as well that this reduction applies to the weak generalization gap, as defined in \eqref{eqn:weak_SP_gap}, as in the case ${\cal Y}=\{\bar{y}\}$ is a singleton:
 \begin{align*} 
\wSPgap{{\cal A}}{f}
 &=\mathbb{E}_{\cal A}[ \sup_{y\in {\cal Y}} \mathbb{E}_{\bS}[f(x(\bS),y)]-\inf_{x\in {\cal X}} \mathbb{E}_{\bS}[f(x, y(\bS))] ]\\
 &=  \mathbb{E}_{\cal A}\mathbb{E}_{\bS}[f(x(\bS),\bar{y})]-\inf_{x\in {\cal X}} f(x, \bar{y})\\
 &=  \mathbb{E}_{{\cal A},\bS}[f(x(\bS),\bar{y})-\inf_{x\in {\cal X}} f(x, \bar{y})],
 \end{align*}
 which is simply the expected optimality gap. Using this reduction, together with a lower bound for DP-SCO \cite{Bassily:2019}, we conclude that
 \begin{proposition}
 Let $n,d\in\NN$, $L,M,D,\varepsilon>0$ and $\delta=o(1/n)$. The class of problems DP-SSP with gradient operators within the class ${\cal M}_{\W}^1(L,M)$, and domain ${\cal W}$ containing an Euclidean ball of diameter $D/2$, satisfies the lower bound
$$ \Omega\Big(MD\Big( \frac{1}{\sqrt n}+\min\Big\{1,\frac{\sqrt{d}}{\varepsilon n}\Big\} \Big)\Big).$$
 \end{proposition}
Next, we study the case of DP-SVI. The situation is more subtle here. Our approach is to first prove a reduction from population weak VI gap to empirical strong VI gap for the case where operators are constant w.r.t.~$w$. In fact, it seems unlikely this reduction works for more general monotone operators, however this suffices for our purposes, as we will later prove a lower bound construction used for DP-ERM \cite{Bassily:2014} leads to a lower bound for strong VI gap with constant  operators.\\
The formal reduction to the empirical version of the problem is presented in the following lemma. Its proof follows closely the reduction from DP-SCO to DP-ERM from \cite{Bassily:2019}. Below, given a dataset $\bS\in {\cal Z}^n$, let ${\cal P}_{\bS}=\frac1n\sum_{\bbeta\in\bS}\delta_{\bbeta}$ be the empirical distribution associated with $\bS$.
\begin{lemma} \label{lem:red_empirical_pop_SVI}
Let ${\cal A}$ be an $(\priva/[4\log(1/\privb)], e^{-\priva}\privb/[8\log(1/\privb)])$-DP algorithm for SVI problems. Then there exists an $(\priva,\privb)$-DP algorithm ${\cal B}$ such that for any empirical VI problem with constant operators, 
\[\EmpVIgap{{\cal B}}{F_{\bS}} \leq \wVIgap{{\cal A}}{F_{{\cal P}_{\bS}}} \qquad (\forall \bS\in {\cal Z}^n).\] 
\end{lemma}
\begin{proof}
Consider the algorithm ${\cal B}$ that does the following: first, it extracts a sample $\mathbf{T}\sim ({\cal P}_{\bS})^n$; next, executes ${\cal A}$ on $\mathbf{T}$; and finally, outputs ${\cal A}(\mathbf{T})$. We claim that this algorithm is $(\priva,\privb)$-DP w.r.t.~$\bS$, which follows easily by bounding the number of repeated examples with high probability, together with the group privacy property applied to ${\cal A}$ (for a more detailed proof, see Appendix C in \cite{Bassily:2019}). Now, given a constant operator $F_{\bbeta}(w)$, let $R(\bbeta)\in\mathbb{R}^d$ be its unique evaluation. Let now $R_{\bS}$ be the unique evaluation of $F_{\bS}$, and given a distribution ${\cal P}$ let $R_{\cal P}$ be the unique evaluation of $F_{{\cal P}}(w)=\mathbb{E}_{\bbeta\sim {\cal P}}[F_{\bbeta}(w)]$. \\
Noting that $\mathbb{E}_{\mathbf{T}}[R_{\mathbf{T}}]=R_{\bS},$ we have that
\begin{align*}
\EmpVIgap{{\cal B}(\bS)}{F_{\bS}}
&= \mathbb{E}_{\cal B}[\sup_{w\in {\cal W}} \langle R_{\bS},{\cal B}(\bS)-w \rangle]\\
&= \mathbb{E}_{{\cal A},\mathbf{T}}[ \langle R_{\bS},{\cal A}(\mathbf{T})\rangle- \inf_{w\in {\cal W}} \langle R_{\bS},w \rangle]\\
&=\mathbb{E}_{{\cal A}} \sup_{w\in {\cal W}} \mathbb{E}_{\mathbf{T}}[ \langle R_{\bS},{\cal A}(\mathbf{T})-w \rangle]\\
&= \wVIgap{{\cal A}}{F_{{\cal P}_{\bS}}},
\end{align*}
where third equality holds since the optimal choice of $w$ is independent of $\mathbf{T}$, and the
last equality holds by definition of the weak gap function and the fact that $\mathbf{T}\sim ({\cal P}_{\bS})^n$.
\end{proof}
 }

\added{Next, we prove a lower bound for the empirical VI problem over constant operators.
\begin{proposition} \label{propos:LB_empirical_VI}
Let $n,d\in\NN$, $L,M,D,\varepsilon>0$ and $2^{-o(n)}\leq \delta \leq o(1/n)$. The class of DP empirical VI problems with constant operators within the class ${\cal M}_{\W}^1(L,M)$, 
and domain ${\cal W}$ containing an Euclidean ball of diameter $D/2$ satisfies the lower bound
$$ \Omega\Big(MD\Big( \min\Big\{1,\frac{\sqrt{d\log(1/\privb)}}{\varepsilon n}\Big\} \Big)\Big).$$
\end{proposition}
\begin{proof}
Consider the following empirical VI problem: $F_{\bbeta}(u)= M\bbeta$, $\W={\cal B}(0,D)$ and dataset $\bS$ with points contained in $\{-1/\sqrt d,+1/\sqrt d\}^d$. Notice that, since the operator in this case is constant, the VI gap coincides with the excess risk of the associated convex optimization problem 
$$ (P)~~~ \min_{u\in \W} \Big\langle\frac{M}{n}\sum_{i\in[n]} \bbeta_i, u\Big\rangle. $$
Indeed, for any $u\in {\cal W}$,
\begin{eqnarray*}
\EmpVIgap{u}{F_\bS} &=& \sup_{v\in {\cal B}(0,D)} \Big\langle \frac{M}{n} \sum_{i\in[n]} \bbeta_i, u- v\Big\rangle
= \Big\langle \frac{M}{n} \sum_{i\in[n]} \bbeta_i, u+ \frac{D\sum_i \bbeta_i}{\|\sum_i \bbeta_i\|}\Big\rangle \\ 
&=& \Big\|\frac{MD}{n}\sum_{i\in[n]} \bbeta_i\Big\|+MD\big\langle \frac{u}{D}, \frac{1}{n}\sum_i \bbeta_i\big\rangle.
\end{eqnarray*}
This, together with the lower bounds on excess risk proved for this problem in \cite[Appendix C]{Bassily:2014} and \cite[Theorem 5.1]{Steinke:2016} 
show that any $(\priva,\privb)$-DP algorithm for this problem must incur in worst-case VI gap $\Omega(MD\min\{1,\frac{\sqrt{d\log(1/\privb)}}{\priva n}\})$,
which proves the result. 
\end{proof}
The two results above provide the claimed lower bound for the weak SVI gap of any differentially private algorithm.
\begin{theorem}
Let $n,d\in\NN$, $L,M,D,\varepsilon>0$ and $2^{-o(n)}\leq \delta \leq o(1/n)$. The class of DP-SVI problems with 
operators within the class ${\cal M}_{\W}^1(L,M)$, 
and domain ${\cal W}$ containing an Euclidean ball of diameter $D/2$ satisfies a lower bound for the weak VI gap 
$$ \tilde\Omega\Big(MD\Big( \frac{1}{\sqrt n}+\min\Big\{1,\frac{\sqrt{d}}{\varepsilon n}\Big\} \Big)\Big).$$
\end{theorem}
Before we prove the result, let us observe that the presented lower bound shows the optimality of our algorithms in the 
range where $M\geq LD$. Obtaining a matching lower bound for any choice of $M,L,D$ is an interesting question, which
unfortunately our proof technique does not address: this is a limitation that the lower bound is based on constant operators, and therefore their Lipschitz constants are always zero. 
\begin{proof}
Let ${\cal A}$ be any algorithm for SVI. By the classical (nonprivate) lower bounds for SVI \cite{Nemirovsky:1983,Juditsky:2011}, we have that the minimax SVI gap attainable is lower bounded by $\Omega(MD/\sqrt n)$. On the other hand, using Lemma \ref{lem:red_empirical_pop_SVI} the accuracy of any $(\priva,\privb)$-DP algorithm for weak SVI gap is lower bounded by the strong gap achieved by $(4\priva \ln(1/\privb),e^{\priva}\tilde O(\privb))$-DP algorithms on empirical VI problems with constant operators. Finally, by Proposition \ref{propos:LB_empirical_VI}, the latter class of problems enjoys a lower bound $\Omega(\min\{1,\sqrt{d\ln(1/[e^{\priva}\tilde O(\privb)])}/[\varepsilon  n\ln(1/\privb)]\})=\tilde\Omega(\min\{1,\sqrt{d}/[n\varepsilon])$, which implies a lower bound on the former class of this order. We conclude by combining both the private and nonprivate
lower bounds established above.
\end{proof}
}

%% file: appendix.tex
\appendix
\section{Proof of Proposition \ref{prop:weak_generalization_SVI}}
Let $\bS' = (\bbeta'_1, \dots, \bbeta'_n)$ be independent of $\bS$. For $i \in [n]$, we denote $\bS^{i} = (\bbeta_1, \dots, \bbeta_{i-1}, \bbeta'_{i},  \bbeta_{i+1}, \dots, \bbeta_n)$. Then, for any $w \in \W$, we have
\begin{align}
  \EE_{\bS}\inprod{F(w)}{\A(\bS) - w} &= \EE_{\bS, \bS'}\frac{1}{n}\sum_{i = 1}^n \inprod{F_{\bbeta'_{i}}(w)}{\A(\bS)-w} \nonumber\\
  &=\EE_{\bS, \bS'}\frac{1}{n}\sum_{i = 1}^n \inprod{F_{\bbeta_{i}}(w)}{\A(\bS^{i})-w} \nonumber\\
  &=\EE_{\bS, \bS'}\frac{1}{n}\sum_{i = 1}^n \inprod{F_{\bbeta_{i}}(w)}{\A(\bS)-w} + \inprod{F_{\bbeta_{i}}(w)}{\A(\bS^{i}) - \A(\bS)} \nonumber\\
  &\le \EE_{\bS, \bS'}\frac{1}{n}\sum_{i = 1}^n \inprod{F_{\bbeta_{i}}(w)}{\A(\bS)-w} + \gnorm{F_{\bbeta_{i}}(w)}{}{}
  \gnorm{\A(\bS^{i}) - \A(\bS)}{}{} \nonumber\\
  &\le \EE_{\bS, \bS'}\frac{1}{n}\sum_{i = 1}^n \inprod{F_{\bbeta_{i}}(w)}{\A(\bS)-w} + M\gnorm{\A(\bS^{i}) - \A(\bS)}{}{} \nonumber \\
  &=\EE_{\bS }\inprod{F_\bS(w)}{\A(\bS) - w} + \EE_{\bS, \bS'} \frac{1}{n}\sum_{i = 1}^n M\gnorm{\A(\bS^{i}) - \A(\bS)}{}{}
  \label{eq:a.s_weak_bound}
\end{align}
Now, taking supremum over $w \in \W$ and taking expectation over $\A$ which is $\delta$-UAS, we have,
\begin{align*}
	\EE_{{\cal A}}[\wVIgap{\A(\bS)}{F}] &\le \EE_{\A}[\sup_{w \in \W}\EE_{\bS}\inprod{F_{\bS}(w) }{\A(\bS) -w}] + \EE_{\bS, \bS'}\frac{1}{n}\sum_{i = 1}^n M \EE_{\A}\gnorm{\A(\bS^i) - \A(\bS)}{}{}\\
	&\le\EE_{\A,\bS}[\sup_{w \in \W}\inprod{F_{\bS}(w) }{\A(\bS) -w}] + \EE_{\bS, \bS'}\frac{1}{n}\sum_{i = 1}^n M \EE_{\A}\gnorm{\A(\bS^i) - \A(\bS)}{}{}\\
	&\le \EE_{\bS}[\EmpVIgap{\A}{F_{\bS}}] + M\delta.
\end{align*}

\section{Operator Extrapolation method \cite{Kotsalis:2020}}
\label{sec:OE_method}
Suppose we want to solve VI problem associated with operator $F_k(\cdot) = F(\cdot)+ \lambda_k(\cdot-w_k)$ whose (unique) solution be $w_{k+1}^*$. It is clear that $F_k$ is an $(L+\lambda_k)$-Lipschitz continuous operator which is $\lambda_k$-strongly monotone as well. Denote $\kappa := \frac{\lambda_k}{L+\lambda_k} + 1$ and consider the following algorithm for solving this problem:
\begin{algorithm}[H]
	\caption{Operator Extrapolation (OE) method}
	\label{alg:alg2}
	\begin{algorithmic}[1]
		\STATE Let $z_0 = z_1 = w_k$ be given.
		\FOR {$t = 1, \dots, T$}
			\STATE $z_{t+1} = \argmin_{w \in \W} \frac{1}{2(L+\lambda_k)} \inprod{F_k(z_t) + \frac{1}{\kappa}\big[F_k(z_t)-F_k(z_{t-1}) \big] }{w} + \frac{1}{2}\gnorm{w-z_t}{}{2}$
		\ENDFOR
	\end{algorithmic}
\end{algorithm}
We have the following convergence guarantee for this algorithm:	$$\gnorm{z_{T}-w^*_{k+1}}{}{2} \le \kappa^{-T}\gnorm{z_1 - w^*_{k+1}}{}{2}$$
In particular, in order to ensue that $\gnorm{z_T-w_{k+1}^*}{}{} \le \frac{\nu}{LD^2+MD+ 2\lambda_kD^2}$, we require $$T = 2\kappa \ln\Big(\frac{LD^2+MD+2\lambda_kD^2}{\nu}\Big) = \frac{2(L+2\lambda_k)}{L+\lambda_k}\ln\Big(\frac{LD^2+MD+2\lambda_kD^2}{\nu}\Big)$$ itearations.